\documentclass[twoside]{article}
\usepackage{amsmath,amssymb}
\usepackage{amsthm}
\usepackage{upref}
\usepackage{graphics,latexsym,mathtools}
\usepackage{color}
\usepackage{enumerate}

\numberwithin{equation}{section}
\allowdisplaybreaks[1]

\theoremstyle{plain}
\newtheorem{theorem}{Theorem}[section]
\newtheorem{proposition}[theorem]{Proposition}
\newtheorem{definition}[theorem]{Definition}
\newtheorem{corollary}[theorem]{Corollary}

\theoremstyle{definition}
\newtheorem{remark}[theorem]{Remark}

\begin{document}

\newcommand{\eq}{equation}
\newcommand{\real}{\ensuremath{\mathbb R}}
\newcommand{\comp}{\ensuremath{\mathbb C}}
\newcommand{\rn}{\ensuremath{{\mathbb R}^n}}
\newcommand{\no}{\ensuremath{\nat_0}}
\newcommand{\ganz}{\ensuremath{\mathbb Z}}
\newcommand{\zn}{\ensuremath{{\mathbb Z}^n}}
\newcommand{\As}{\ensuremath{A^s_{p,q}}}
\newcommand{\Bs}{\ensuremath{B^s_{p,q}}}
\newcommand{\Fs}{\ensuremath{F^s_{p,q}}}
\newcommand{\nat}{\ensuremath{\mathbb N}}
\newcommand{\Om}{\ensuremath{\Omega}}
\newcommand{\di}{\ensuremath{{\mathrm{d}}}}
\newcommand{\vp}{\ensuremath{\varphi}}
\newcommand{\hra}{\ensuremath{\hookrightarrow}}
\newcommand{\supp}{\ensuremath{\mathrm{supp \,}}}
\newcommand{\ve}{\ensuremath{\varepsilon}}
\newcommand{\vk}{\ensuremath{\varkappa}}
\newcommand{\vr}{\ensuremath{\varrho}}
\newcommand{\id}{\ensuremath{\mathrm{id}}}
\newcommand{\Rea}{\ensuremath{\mathrm{Re\,}}}
\newcommand{\Ima}{\ensuremath{\mathrm{Im\,}}}
\newcommand{\loc}{\ensuremath{\mathrm{loc}}}
\newcommand{\wh}{\ensuremath{\widehat}}
\newcommand{\wt}{\ensuremath{\widetilde}}
\newcommand{\ol}{\ensuremath{\overline}}
\newcommand{\os}{\ensuremath{\overset}}
\newcommand{\Nc}{\ensuremath{\mathcal N}}
\newcommand{\SRn}{\mathcal{S}(\rn)}
\newcommand{\SpRn}{\mathcal{S}'(\rn)}
\newcommand{\Dd}{\mathrm{D}}
\newcommand{\Ft}{\mathcal{F}}
\newcommand{\Fti}{\mathcal{F}^{-1}}
\newcommand{\bli}{\begin{enumerate}[{\upshape\bfseries (i)}]}
  \newcommand{\eli}{\end{enumerate}}
\newcommand{\nd}{\ensuremath {{n}}} 
\newcommand{\Ae}{\ensuremath{A^{s_1}_{p_1,q_1}}}  
\newcommand{\Az}{\ensuremath{A^{s_2}_{p_2,q_2}}}  
\newcommand{\open}[1]{\smallskip\noindent\fbox{\parbox{\textwidth}{\color{blue}\bfseries\begin{center}
      #1 \end{center}}}\\ \smallskip}
\newcommand{\red}[1]{{\color{red}#1}}
\newcommand{\blue}[1]{{\color{blue}#1}}
\renewcommand{\blue}[1]{{#1}}
\newcommand{\purple}[1]{{\color{purple}#1}}
\newcommand{\cyan}[1]{{\color{cyan}#1}}
\newcommand{\green}[1]{{\color{green}#1}}
\newcommand{\unterbild}[1]{{\noindent\refstepcounter{figure}\upshape\bfseries
    Figure {\thefigure}{\label{#1}}}
}%
\newcommand{\ignore}[1]{}

\title{Nuclear Fourier transforms}

\author{Dorothee D. Haroske, Leszek Skrzypczak and Hans Triebel}
\date{}
\maketitle

\begin{abstract}
{
The paper deals with the problem under which conditions for the parameters $s_1,s_2\in\real$, $1\leq p,q_1,q_2\leq\infty$ the Fourier transform $\Ft$ is a nuclear mapping from $A^{s_1}_{p,q_1}(\rn)$ into $A^{s_2}_{p,q_2}(\rn)$, where $A\in\{B,F\}$ stands for a space of Besov or Triebel-Lizorkin type, and $n\in\nat$. It extends the recent paper \cite{T21} where the compactness of $\Ft$ acting in the same type of spaces was studied.\\
  }

\end{abstract}

\section{Introduction}   \label{S1}

Let $\Ft$,
\begin{\eq}   \label{1.1}
\big(\Ft \vp\big)(\xi) = (2 \pi)^{-n/2} \int_{\rn} e^{-i x \xi} \, \vp (x) \, \di x, \qquad \vp \in \SRn, \quad \xi \in \rn,
\end{\eq}
be the classical Fourier transform, extended  in the usual way to $\SpRn$, $n\in \nat$. The mapping properties
\begin{\eq}   \label{1.2}
\Ft\SRn = \SRn, \qquad \Ft \SpRn = \SpRn,
\end{\eq}
and
\begin{\eq}   \label{1.3}
\Ft:  \ L_p (\rn) \hra L_{p'} (\rn), \quad  1\le p \le 2, \quad \frac{1}{p} + \frac{1}{p'} =1, \quad \Ft L_2 (\rn) = L_2 (\rn),
\end{\eq}
are cornerstones of Fourier {analysis}. These basic assertions have been complemented in \cite{T21} covering in particular the following
observation. Let $\As (\rn)$ with
\begin{\eq}   \label{1.4}
A \in \{B,F \}, \quad 1<p,q_1,q_2 < \infty \quad \text{and} \quad s_1 \in \real, \quad s_2 \in \real,
\end{\eq}
be the usual function spaces of Besov and Triebel-Lizorkin type. We denote 
\begin{\eq}   \label{1.5}
d^n_p = 2n \left( \frac{1}{p} - \frac{1}{2} \right), \quad 1<p<\infty, \quad n\in \nat,
\end{\eq}
and introduce
\begin{\eq}   \label{1.6}
\tau^{n+}_p = \max (0, d^n_p) \quad \text{and} \quad \tau^{n-}_p = \min (0, d^n_p).
\end{\eq}
Then
\begin{\eq}   \label{1.7}
\Ft: \quad A^{s_1}_{p,q_1} (\rn) \hra A^{s_2}_{p,q_2} (\rn)
\end{\eq}
is compact if
\begin{\eq}   \label{1.8}
\text{both} \quad s_1 > \tau^{n+}_p \quad \text{and} \quad s_2 < \tau^{n-}_p\ .
\end{\eq}
{If  (independently)}
\begin{\eq}   \label{1.9}
\text{either} \quad s_1 < \tau^{n+}_p \qquad \text{or} \quad s_2 > \tau^{n-}_p,
\end{\eq}
{then there is no continuous embedding of type \eqref{1.7}.} {We refer to Figure~\ref{fig-1} below for some diagram.} 
It was one of the main aims of \cite{T21} to deal with the degree of compactness of $\Ft$ in \eqref{1.7} in case of \eqref{1.8}, expressed in terms of 
entropy numbers. In the present paper we ask for conditions ensuring that the mapping $\Ft$ in \eqref{1.7} is nuclear. Recall that a
linear continuous mapping $T: \ A\hra B$ from the {\ignore{complex} Banach space $A$ into the \ignore{complex} Banach space}  $B$ is called nuclear if it
can be represented as
\begin{\eq}   \label{1.10}
Tf = \sum^\infty_{k=1} (f, a_k') b_k, \qquad \{a'_k \} \subset A', \quad \{b_k \} \subset B,
\end{\eq}
such that $\sum^\infty_{k=1} \| a'_k \, | A'\| \cdot \| b_k \, | B \|$ is finite, where $A'$ is the dual of $A$. In particular, any
nuclear mapping is compact. We refer to Section~\ref{S3.1} for further details and some history of the topic.

Our main result is Theorem~\ref{T3.1} characterising in particular under which conditions the compact mapping
\eqref{1.7} with \eqref{1.8} is nuclear. {We refer to Figure~\ref{fig-3} below for some illustration.} 

The paper is organised as follows. In Section~\ref{S2} we collect definitions and some ingredients. This includes wavelet characterisations and weighted generalisations
$\As (\rn, w_\alpha)$ of the above unweighted spaces $\As (\rn)$, where the function $w_\alpha (x) = (1+|x|^2)^{\alpha/2}$, $\alpha \in \real$, is a so-called  `admissible' weight. In Section~\ref{S3}
we recall first some already known properties  about nuclear embeddings between these spaces and prove Theorem~\ref{T3.1}. This will be complemented by related assertions for some limiting cases. Finally, in Section~\ref{S4} we collect some more or less immediate consequences when $\Ft$ is considered as mapping between weighted spaces of type $\As (\rn, w_\alpha)$.

\section{Definitions and ingredients}   \label{S2}
\subsection{Definitions and some basic properties}   \label{S2.1}
We use standard notation. Let $\nat$ be the collection of all natural numbers and $\no = \nat \cup \{0 \}$. Let $\rn$ be {the} Euclidean $n$-space where
$n\in \nat$. Put $\real = \real^1$.
Let $\SRn$ be the Schwartz space of all complex-valued rapidly decreasing infinitely differentiable functions on $\rn$ and let $\SpRn$ be the dual space {consisting} of all tempered distributions on $\rn$.
Furthermore, $L_p (\rn)$ with $0< p \le \infty$, is the standard complex quasi-Banach space with respect to the Lebesgue measure, quasi-normed by
\begin{\eq}   \label{2.1}
\| f \, | L_p (\rn) \| = \Big( \int_{\rn} |f(x)|^p \, \di x \Big)^{1/p}
\end{\eq}
with the obvious modification if $p=\infty$.  
As usual, $\ganz$ is the collection of all integers; and $\zn$, $n\in \nat$, denotes the
lattice of all points $m= (m_1, \ldots, m_n) \in \rn$ with $m_k \in \ganz$. 
\smallskip~

If $\vp \in \SRn$, then
\begin{\eq}  \label{2.2}
\wh{\vp} (\xi) = (\Ft \vp)(\xi) = (2\pi )^{-n/2} \int_{\rn} e^{-ix \xi} \vp (x) \, \di x, \qquad \xi \in  \rn,
\end{\eq}
denotes the Fourier transform of \vp. As usual, $\Fti \vp$ and $\vp^\vee$ stand for the inverse Fourier transform, given by the right-hand side of
\eqref{2.2} with $i$ in place of $-i$. Here $x \xi$ stands for the scalar product in \rn. Both $\Ft$ and $\Fti$ are extended to $\SpRn$ in the
standard way. Let $\vp_0 \in \SRn$ with
\begin{\eq}   \label{2.3}
\vp_0 (x) =1 \ \text{if $|x|\le 1$} \quad \text{and} \quad \vp_0 (x) =0 \ \text{if $|x| \ge 3/2$},
\end{\eq}
and let
\begin{\eq}   \label{2.4}
\vp_k (x) = \vp_0 (2^{-k} x) - \vp_0 (2^{-k+1} x ), \qquad x\in \rn, \quad k\in \nat.
\end{\eq}
Since
\begin{\eq}   \label{2.5}
\sum^\infty_{j=0} \vp_j (x) =1 \qquad \text{for} \quad x\in \rn,
\end{\eq}
$\vp =\{ \vp_j \}^\infty_{j=0}$ forms a smooth dyadic resolution of unity. The entire analytic functions $(\vp_j \wh{f} )^\vee (x)$ make sense pointwise in $\rn$ for any $f\in \SpRn$. 
\smallskip~

\begin{definition}   \label{D2.1}
Let $\vp = \{ \vp_j \}^\infty_{j=0}$ be the above dyadic resolution  of unity. Let $s\in \real$, $0<q\leq\infty$. 
\bli
\item  
Let  $0<p \le \infty$. 
Then $\Bs (\rn)$ is the collection of all $f \in \SpRn$ such that
\begin{\eq}   \label{2.6}
\| f \, | \Bs (\rn) \|_{\vp} = \Big( \sum^\infty_{j=0} 2^{jsq} \big\| (\vp_j \wh{f})^\vee \, | L_p (\rn) \big\|^q \Big)^{1/q}
\end{\eq}
is finite $($with the usual modification if $q= \infty)$. 
\item Let $0<p<\infty$. Then $\Fs (\rn)$ is the collection of all $f\in \SpRn$ such that
\begin{\eq}   \label{2.7}
\| f \, | \Fs (\rn) \|_{\vp} 
= \Big\| \Big( \sum^\infty_{j=0} 2^{jsq} \big| (\vp_j \wh{f})^\vee (\cdot) \big|^q \Big)^{1/q} \, | L_p (\rn) \Big\|
\end{\eq}
is finite $($with the usual modification if $q=\infty)$.
\eli
\end{definition}

\begin{remark}   \label{R2.2}
These well--known inhomogeneous spaces are independent of the above resolution of unity $\vp$ according to \eqref{2.3}--\eqref{2.5} in the sense of 
equivalent quasi--norms. This justifies the omission of the subscript $\vp$ in \eqref{2.6}, \eqref{2.7} in the sequel. Let us mention here, in particular, the series of monographs \cite{T83,T92,T06,T20}, where also one finds further historical references, explanations and discussions. \blue{The above restriction to $p<\infty$ in case of $\Fs{(\rn)}$ is the usual one, {though   many important results  could be extended to $F^s_{\infty,q}{(\rn)}$, cf. \cite{T20} for the definition and  properties of the spaces as well as  historical remarks. Here we stick to the above setting.} 
\ignore{	though in \cite{T20} many important results could be extended to $F^s_{\infty,q}{(\rn)}$. But here we stick to the above setting. }}

As usual we write $\As (\rn)$, $A \in \{B,F \}$, if the related
assertion applies equally to the $B$--spaces $\Bs (\rn)$ and the $F$--spaces $\Fs (\rn)$. We deal mainly with the $B$--spaces. The 
$F$--spaces can often be incorporated in related assertions using the embedding
\begin{\eq}   \label{2.8}
B^s_{p, \min(p,q)} (\rn) \hra \Fs (\rn) \hra B^s_{p, \max(p,q)} (\rn),
\end{\eq}
$s\in \real$, $0<p<\infty$, $0<q \le \infty$. Let
\begin{\eq}   \label{2.9}
w_\alpha (x) = (1 + |x|^2 )^{\alpha/2}, \qquad x\in \rn, \quad \alpha \in \real.
\end{\eq}
Then $I_\alpha$,
\begin{\eq}   \label{2.10}
I_\alpha: \quad f \mapsto \big( w_\alpha \wh{f}\, \big)^\vee = \big( w_\alpha f^\vee \big)^\wedge, \qquad f\in \SpRn, \quad \alpha \in\real,
\end{\eq}
is a lift in the spaces $\As (\rn)$, $s\in \real$, $0<p<\infty$, $0<q \le \infty$, mapping $\As (\rn)$ isomorphically onto 
$A^{s-\alpha}_{p,q} (\rn)$,
\begin{\eq}   \label{2.11}
I_\alpha \As (\rn) = A^{s-\alpha}_{p,q} (\rn), \quad \|(w_\alpha \wh{f} \, )^\vee | A^{s-\alpha}_{p,q} (\rn) \| \sim \|f \, | \As 
(\rn) \|,
\end{\eq}
equivalent quasi--norms, see \cite[Theorem 1.22, p.\,16]{T20} and the {references given there}. Of interest for us will be the Sobolev spaces
\begin{\eq}   \label{2.12}
H^s_p (\rn) = F^s_{p,2} (\rn), \qquad s\in \real, \quad 1<p<\infty,
\end{\eq}
their Littlewood--Paley characterisations and
\begin{\eq}  \label{2.13}
I_s H^s_p (\rn) = L_p (\rn), \quad \| (w_s \wh{f})^\vee | L_p (\rn) \| = \|f \, | H^s_p (\rn) \|.
\end{\eq}
\end{remark}

\begin{remark}   \label{R2.3}
For our arguments below we need the weighted counterparts of the spaces $\As (\rn)$ as introduced in Definition~\ref{D2.1}. Let $s,p,q$
be as there and let $w_\alpha$ be the weight according to \eqref{2.9}. Then $\As (\rn, w_\alpha)$ is the collection of all $f\in 
\SpRn$ such that \eqref{2.6}, \eqref{2.7} with $L_p (\rn, w_\alpha)$ in place of $L_p (\rn)$ is finite. Here $L_p (\rn, w_\alpha)$
is the complex quasi--Banach space quasi--normed by
\begin{\eq} \label{2.14}
\| f \, | L_p (\rn, w_\alpha) \| = \| w_\alpha f \, | L_p (\rn) \|, \qquad 0<p \le \infty, \quad \alpha \in \real.
\end{\eq}
These spaces have some remarkable properties  which will be of some use for us later on, see also \cite{HT1} and \cite[Sect.~4.2]{ET}. In particular, for all spaces $f \mapsto 
w_\alpha f$ is an isomorphic mapping,
\begin{\eq}   \label{2.15}
\| w_\alpha f \, | \As (\rn) \| \sim \|f \, | \As (\rn, w_\alpha) \|, \qquad \alpha \in \real,
\end{\eq}
and for all spaces the lifting \eqref{2.11} can be extended from the unweighted spaces to their weighted counterparts,
\begin{\eq}   \label{2.16}
\begin{aligned}
I_\alpha \As (\rn, w_\beta) &= A^{s-\alpha}_{p,q} (\rn, w_\beta), \\
 \|(w_\alpha \wh{f} )^\vee | A^{s-\alpha}_{p,q} (\rn, w_\beta) \|
&\sim \| f\, | \As (\rn, w_\beta) \|,
\end{aligned}
\end{\eq}
$\alpha \in \real$, $\beta \in \real$. Both (substantial) assertions are covered by \cite[Theorem 6.5, pp.\,265--266]{T06} and the
{references given there}. \blue{Note that weights of type $w_\alpha$ given by \eqref{2.9} are also special Muckenhoupt weights when $\alpha>-n$, that is,  $w_\alpha \in \mathcal{A}_\infty$ if $\alpha>-n$. }
\end{remark}

\subsection{Wavelet characterisations}    \label{S2.2}
Our arguments below rely on wavelet representations for some (unweighted) $B$--spaces. We collect what we need following \cite[Section
1.2.1, pp.\,7--10]{T20}. There one finds explanations and related references. \blue{Let us, in particular, refer to the standard monographs for this topic  \cite{Dau92,Mal98,Mey92,Woj97}. A short summary can also be found in \cite[Sect.~1.7]{T06}}. 

As usual, $C^{u} (\real)$ with $u\in
\nat$ collects all bounded complex-valued continuous functions on $\real$ having continuous bounded derivatives up to order $u$ inclusively. Let
\begin{\eq}   \label{2.17}
\psi_F \in C^{u} (\real), \qquad \psi_M \in C^{u} (\real), \qquad u \in \nat,
\end{\eq}
be  real compactly supported Daubechies wavelets with
\begin{\eq}   \label{2.18}
\int_{\real} \psi_M (x) \, x^v \, \di x =0 \qquad \text{for all $v\in \no$ with $v<u$.}
\end{\eq}
Let $n\in \nat$ and let
\begin{\eq}   \label{2.19}
G = (G_1, \ldots, G_n) \in G^0 = \{F,M \}^n
\end{\eq}
which means that $G_r$ is either $F$ or $M$. Furthermore, let
\begin{\eq}   \label{2.20}
G= (G_1, \ldots, G_n) \in G^j = \{F, M \}^{n*}, \qquad j \in \nat,
\end{\eq}
which means that $G_r$ is either $F$ or $M$, where $*$ indicates that at least one of the components of $G$ must be an $M$. Let
\begin{\eq}   \label{2.21}
\psi^j_{G,m} (x) = \prod^n_{l=1} \psi_{G_l} \big(2^j x_l -m_l \big), \qquad G\in G^j, \quad m \in \zn,
\end{\eq}
$x\in \rn$, where (now) $j \in \no$. Then we may assume that 
\begin{\eq}   \label{2.22}
 \big\{ 2^{jn/2} \psi^j_{G,m}: \ j \in \no, \ G\in G^j, \ m \in \zn \big\}
\end{\eq}
is an orthonormal basis in $L_2 (\rn)$. Let 
\begin{\eq}   \label{2.23}
1\le p,q {\leq} \infty \qquad \text{and} \qquad s\in \real.
\end{\eq}


Let $u\in \nat$ such that $|s| <u$. Then $f\in \Bs (\rn)$ can be represented as
\begin{\eq}   \label{2.24}
f = \sum_{j=0}^\infty \sum_{G\in G^j} \sum_{m \in \zn} 2^{jn} (f, \psi^j_{G,m} ) \, \psi^j_{G.m}
\end{\eq}
with
\begin{\eq}   \label{2.25}
\|f \, | \Bs (\rn) \| \sim \bigg( \sum_{j=0}^\infty 2^{j(s- \frac{n}{p})q} \sum_{G\in G^j} \Big( \sum_{m\in \zn} 2^{jnp} \big|
(f, \psi^j_{G,m} ) \big|^p \Big)^{q/p} \bigg)^{1/q},
\end{\eq}
where the equivalence constants are independent of $f$, {with the usual modification if $\max(p,q)=\infty$}. In particular, the series in \eqref{2.24} converges unconditionally in {$\SpRn$ and unconditionally even in $\Bs(\rn)$ if $\max(p,q)<\infty$.} Furthermore \eqref{2.22} is a basis in $\Bs (\rn)$ {if $\max(p,q)<\infty$}. From \eqref{2.25} and \eqref{2.8} it follows that
\begin{\eq}   \label{2.26}
\| \psi^j_{G.m} \, | \As (\rn) \| \sim 2^{j(s- \frac{n}{p})}, \qquad j\in \no, \quad m\in \zn, \quad G\in G^j,
\end{\eq}
$1 \le p,q {\leq} \infty$, $s\in \real$, $A\in \{B,F \}$ {(with $p<\infty$ when $A=F$)},  where the equivalence constants can be chosen independently of $j,G,m$.

\subsection{Mappings}   \label{S2.3}
We recall some mapping properties of the Fourier transform $\Ft$ obtained in \cite{T21}. This covers also (more or less) what has already
been said in the Introduction, \eqref{1.4}--\eqref{1.9}.

Let 
$n\in \nat$, $1<p<\infty$ and $s\in \real$. We use the notation
\begin{\eq}   \label{2.27}
d^n_p = 2n \left( \frac{1}{p} - \frac{1}{2} \right), \qquad 1<p < \infty, \quad n\in \nat,
\end{\eq}
and define
\begin{\eq}  \label{2.28}
\tau^{n+}_p = \max (0, d^n_p) \qquad \text{and} \qquad \tau^{n-}_p = \min (0, d^n_p).
\end{\eq}
%
We denote by
\[
X^s_p (\rn) =
\begin{cases}
L_p (\rn) &\text{if $2\le p<\infty$, $s=0$}, \\
{B^s_{p,p}} (\rn) & \text{if $2 \le p <\infty$, $s>0$}, \\
{B^s_{p,p}} (\rn) &\text{if $1<p \le 2$, $s \ge d^n_p$},
\end{cases}
\]
and
\[
Y^s_p (\rn) =
\begin{cases}
{B^s_{p,p}} (\rn) &\text{if $2 \le p <\infty$, $s \le d^n_p$}, \\
{B^s_{p,p}} (\rn) &\text{if $1<p \le 2$, $s<0$}, \\
L_p (\rn) &\text{if $1<p\le 2$, $s=0$}.
\end{cases}
\]
For convenience, we have sketched in the usual $(\frac1p,s)$-diagram in Figure~\ref{fig-1} below  the corresponding areas for the definition of $X^s_p$ and $Y^s_p$. Here any space $\Bs(\rn)$ is indicated by its parameters $s$ and $p$, neglecting $q$.

\noindent\begin{minipage}{\textwidth}
  ~\hfill\begin{picture}(0,0)%
\includegraphics{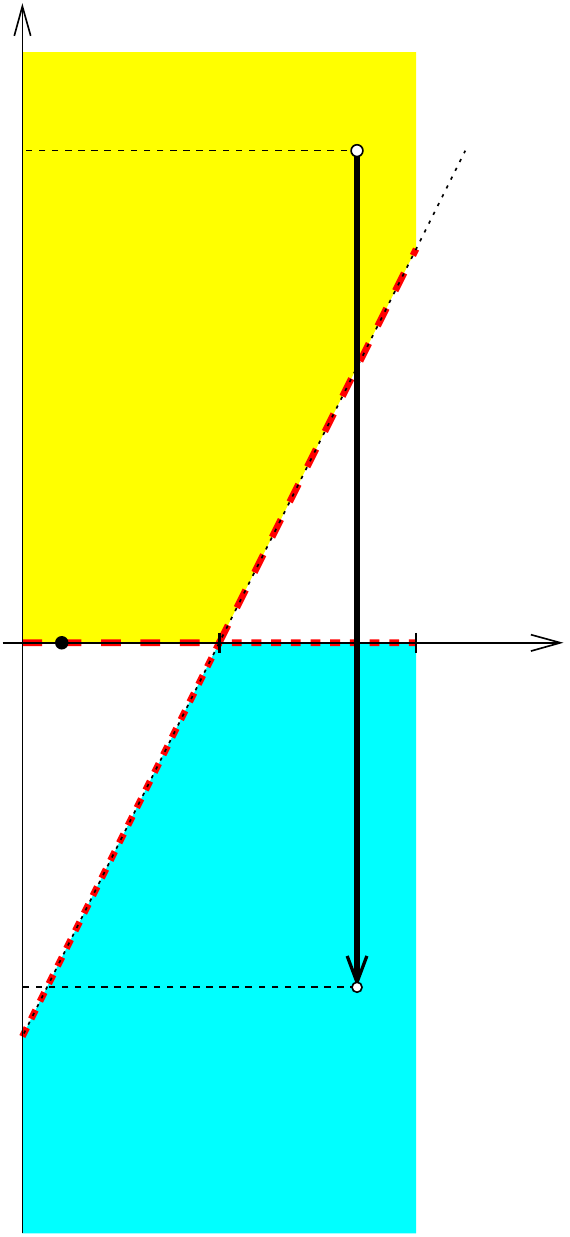}%
\end{picture}%
\setlength{\unitlength}{4144sp}%
\begingroup\makeatletter\ifx\SetFigFont\undefined%
\gdef\SetFigFont#1#2#3#4#5{%
  \reset@font\fontsize{#1}{#2pt}%
  \fontfamily{#3}\fontseries{#4}\fontshape{#5}%
  \selectfont}%
\fi\endgroup%
\begin{picture}(2589,5649)(349,-5023)
\put(2386,-376){\makebox(0,0)[lb]{\smash{{\SetFigFont{10}{12.0}{\familydefault}{\mddefault}{\updefault}{\color[rgb]{0,0,0}$s=d^n_p$}%
}}}}
\put(856,-3436){\makebox(0,0)[lb]{\smash{{\SetFigFont{10}{12.0}{\familydefault}{\mddefault}{\updefault}{\color[rgb]{0,0,0}$s=\tau^{n-}_p$}%
}}}}
\put(2026,-1546){\makebox(0,0)[lb]{\smash{{\SetFigFont{10}{12.0}{\familydefault}{\mddefault}{\updefault}{\color[rgb]{0,0,0}$\Ft$}%
}}}}
\put(1801,-1366){\makebox(0,0)[rb]{\smash{{\SetFigFont{10}{12.0}{\familydefault}{\mddefault}{\updefault}{\color[rgb]{0,0,0}$s=\tau^{n+}_p$}%
}}}}
\put(2251,-2161){\makebox(0,0)[b]{\smash{{\SetFigFont{10}{12.0}{\familydefault}{\mddefault}{\updefault}$1$}}}}
\put(1261,-2161){\makebox(0,0)[b]{\smash{{\SetFigFont{10}{12.0}{\familydefault}{\mddefault}{\updefault}$\frac12$}}}}
\put(631,-2221){\makebox(0,0)[b]{\smash{{\SetFigFont{10}{12.0}{\familydefault}{\mddefault}{\updefault}$L_p$}}}}
\put(1981,-4111){\makebox(0,0)[b]{\smash{{\SetFigFont{10}{12.0}{\familydefault}{\mddefault}{\updefault}$B^ {s_2}_{p,p}$}}}}
\put(901,-4561){\makebox(0,0)[lb]{\smash{{\SetFigFont{10}{12.0}{\familydefault}{\mddefault}{\updefault}$Y^s_p$}}}}
\put(1981, 74){\makebox(0,0)[b]{\smash{{\SetFigFont{10}{12.0}{\familydefault}{\mddefault}{\updefault}$B^{s_1}_{p,p}$}}}}
\put(901,-601){\makebox(0,0)[b]{\smash{{\SetFigFont{10}{12.0}{\familydefault}{\mddefault}{\updefault}$X^s_p$}}}}
\put(406,-4201){\makebox(0,0)[rb]{\smash{{\SetFigFont{10}{12.0}{\familydefault}{\mddefault}{\updefault}$-n$}}}}
\put(2701,-2491){\makebox(0,0)[lb]{\smash{{\SetFigFont{10}{12.0}{\familydefault}{\mddefault}{\updefault}$\frac1p$}}}}
\put(406,389){\makebox(0,0)[rb]{\smash{{\SetFigFont{10}{12.0}{\familydefault}{\mddefault}{\updefault}$s$}}}}
\put(406,-106){\makebox(0,0)[rb]{\smash{{\SetFigFont{10}{12.0}{\familydefault}{\mddefault}{\updefault}$s_1$}}}}
\put(406,-3931){\makebox(0,0)[rb]{\smash{{\SetFigFont{10}{12.0}{\familydefault}{\mddefault}{\updefault}$s_2$}}}}
\put(1936,-2161){\makebox(0,0)[rb]{\smash{{\SetFigFont{10}{12.0}{\familydefault}{\mddefault}{\updefault}$\frac1p$}}}}
\end{picture}%
\hfill~\\
~\hspace*{\fill}\unterbild{fig-1}
\end{minipage}
\smallskip~

We collect what is already known about the continuity and compactness of the map $\Ft : X^{s_1}_p(\rn) \hra Y^{s_2}_p(\rn)$.

\begin{theorem}[{\cite{T21}}]\label{Thm-comp}
  Let $1<p<\infty$, $s_1{\in\real}, s_2\in\real$ and $\tau^{n+}_p$, $\tau^{n-}_p$ be given by \eqref{2.28} with \eqref{2.27}.
  \bli
\item
    Then
\begin{\eq}   \label{2.31}
\Ft: \quad X^{s_1}_p (\rn) \hra Y^{s_2}_p (\rn) \quad \text{with $s_1 \ge \tau^{n+}_p$ and $s_2 \le \tau^{n-}_p$}
\end{\eq}
is continuous. This mapping is even compact if, and only if, both $s_1 > \tau^{n+}_p$ and $s_2 < \tau^{n-}_p$.
\item
  Furthermore, if 
{there is a continuous mapping}
\begin{\eq} \label{2.32}
  \Ft: \quad B^{s_1}_{p,p} (\rn) \hra B^{s_2}_{p,p} (\rn),
\end{\eq}
{then} 
both $s_1 \ge \tau^{n+}_p$ and $s_2 \le \tau^{n-}_p$.
\eli
\end{theorem}

\begin{remark}
  These assertions are covered by \cite[Theorem~3.2, Corollary~3.3]{T21}. There one also finds results about the entropy numbers $e_k(\Ft)$, $k\in\nat$, of $\Ft$  which further characterise the `degree of compactness'. 
\end{remark}

\begin{corollary}\label{F-comp-cor}
  Let $1<p<\infty$, $0<q_1,q_2\leq\infty$, $s_1, s_2\in\real$. Let $A\in \{B, F\}$. Then
  \begin{\eq}   \label{2.33}
\Ft: \quad A^{s_1}_{p,q_1}(\rn) \hra A^{s_2}_{p,q_2} (\rn) 
\end{\eq}
is compact if both $s_1 > \tau^{n+}_p$ and $s_2 < \tau^{n-}_p$.\\
{If $s_1 < \tau^{n+}_p$ or $s_2 > \tau^{n-}_p$}, {then there is no continuous map \eqref{2.33}}.
\end{corollary}

\begin{proof}
This is an immediate consequence of Theorem~\ref{Thm-comp} together with elementary embeddings like \eqref{2.8} and $A^{s+\ve}_{p,q_1}(\rn)\hra B^s_{p,p}(\rn) \hra A^{s-\ve}_{p,q_1}(\rn)$ for arbitrary $\ve>0$. 
\end{proof}

The above result shows that $s_1 = \tau^{n+}_p$ and $s_2 = \tau^{n-}_p$ are natural barriers if one wishes to study continuous and
compact mappings of type \eqref{2.33}. The observation justifies \eqref{1.7}--\eqref{1.9}.
It also implies that related restrictions for $s_1$ and $s_2$ in what follows are natural. 
%
%
{We complement the above assertion in Section~\ref{S3.3} where we shall also deal with the limiting cases $p=1$ and $p=\infty$.}


\section{Nuclear mappings}    \label{S3}
\subsection{Preliminaries}   \label{S3.1}
A linear continuous mapping $T: A \hra B$ from the (complex) Banach space $A$ into the (complex) Banach space $B$ is called nuclear if it can be represented as
\begin{\eq}   \label{3.1}
Tf = \sum_{k=1}^\infty (f, a_k' )\, b_k, \qquad \{a_k' \} \subset A', \quad \{b_k \} \subset B,
\end{\eq}
such that $\sum_{k=1}^\infty \|a_k' |A' \| \cdot \| b_k |B \|$ is finite. Here $A'$ is the dual of $A$. Then
\begin{\eq}   \label{3.2}
\| T \, | \Nc(A,B)\| = \inf \sum_{k=1}^\infty \|a_k' |A' \| \cdot \| b_k |B \|
\end{\eq}
is the related nuclear norm, where the infimum is taken over all representations \eqref{3.1}. In particular any nuclear mapping is
compact. The collection of all nuclear mappings
between complex Banach spaces is a symmetric operator ideal. Here symmetric means that $T': \, B' \hra A'$ is nuclear if $T: \, A
\hra B$ is nuclear, \cite[8.2.6, p.\,108]{Pie78}, \cite[p.\,280]{Pie07}. 

\begin{remark}
  Grothendieck introduced the concept of nuclearity in \cite{grothendieck} more than 60 years ago. It provided the basis for many famous developments in functional analysis afterwards, we refer to \cite{Pie78}, and, in particular, to \cite{Pie07} for further historic details.  In Hilbert spaces $H_1,H_2$, the nuclear operators $\mathcal{N}(H_1,H_2)$ coincide with the trace class $S_1(H_1,H_2)$, consisting of those $T$ with singular numbers $(s_k(T))_{k\in\nat} \in \ell_1$. It is well known from the remarkable Enflo result \cite{enflo} that there are compact operators between Banach spaces which cannot be approximated by finite-rank operators.
 This led to a number of -- meanwhile well-established  and famous -- methods to circumvent this difficulty and find alternative ways to `measure' the compactness or `degree' of compactness of an operator, e.g. the asymptotic behaviour of its approximation or entropy numbers. In all these problems, the decomposition
of a given compact operator into a series is an essential proof technique. It turns out that in many of the recent contributions \cite{T17,CoDoKu,CoEdKu,HaS20,HaLeoSk} studying nuclearity, a key tool in the arguments are new decomposition techniques as well, adapted to the different spaces. This is also our intention now.
\end{remark}

In addition to the tools described above we will rely on the following two observations about nuclear embeddings between function
spaces.

Let $\Om$ be a bounded Lipschitz domain in \rn, $n\in \nat$, (bounded interval if $n=1$). Then $\As (\Om)$ is, as usual, the 
restriction of the spaces $\As (\rn)$ as introduced in Definition~\ref{D2.1} and Remark~\ref{R2.2}. 

\begin{proposition}\label{nuc-dom}
Let
\begin{\eq}  \label{3.3}
1<p_1, p_2, q_1 , q_2 <\infty \qquad \text{and} \qquad s_1 \in \real, \quad s_2 \in \real. 
\end{\eq}
\bli
\item
The embedding
\begin{\eq}   \label{3.4}
\id: \quad A^{s_1}_{p_1, q_1} (\Om) \hra A^{s_2}_{p_2,q_2} (\Om)
\end{\eq}
is compact, if, and only if,
\begin{equation}\label{id_Omega-comp}
s_1-s_2 > n \max\left(\frac{1}{p_1}-\frac{1}{p_2},0\right).
\end{equation}
\item
The embedding \eqref{3.4} 
is nuclear if, and only if,
\begin{\eq}   \label{3.5}
s_1 - s_2 > n -n \max \left( \frac{1}{p_2} - \frac{1}{p_1}, 0 \right).
\end{\eq}
\eli
\end{proposition}
The result (ii) can be found in \cite[Theorem, p.\,3039]{T17}, clarifying some limiting cases compared with what was already known before, cf. \cite{Pie-r-nuc,PiTri}. {In \cite{HaS20} we also dealt with the situations $p=1$ and $p=\infty$. Let us briefly illustrate the situation in the diagram below.}\\

\noindent\begin{minipage}{\textwidth}
  ~\hfill\begin{picture}(0,0)%
\includegraphics{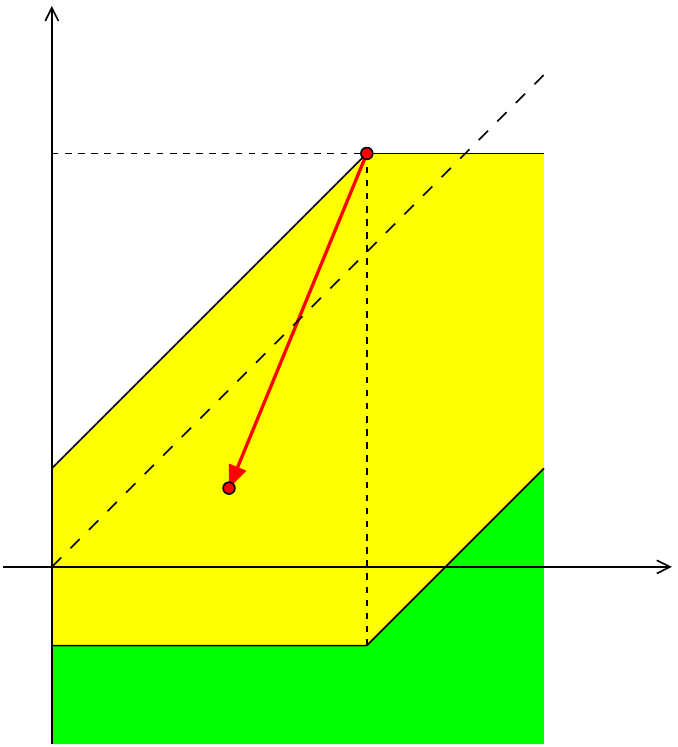}%
\end{picture}%
\setlength{\unitlength}{4144sp}%
\begingroup\makeatletter\ifx\SetFigFont\undefined%
\gdef\SetFigFont#1#2#3#4#5{%
  \reset@font\fontsize{#1}{#2pt}%
  \fontfamily{#3}\fontseries{#4}\fontshape{#5}%
  \selectfont}%
\fi\endgroup%
\begin{picture}(3084,3399)(1114,-4483)
\put(1261,-4066){\makebox(0,0)[rb]{\smash{{\SetFigFont{10}{12.0}{\familydefault}{\mddefault}{\updefault}{\color[rgb]{0,0,0}$s_1-\nd$}%
}}}}
\put(1306,-3256){\makebox(0,0)[rb]{\smash{{\SetFigFont{10}{12.0}{\familydefault}{\mddefault}{\updefault}{\color[rgb]{0,0,0}$s_1-\frac{\nd}{p_1}$}%
}}}}
\put(1306,-1816){\makebox(0,0)[rb]{\smash{{\SetFigFont{10}{12.0}{\familydefault}{\mddefault}{\updefault}{\color[rgb]{0,0,0}$s_1$}%
}}}}
\put(1306,-1231){\makebox(0,0)[rb]{\smash{{\SetFigFont{10}{12.0}{\familydefault}{\mddefault}{\updefault}{\color[rgb]{0,0,0}$s$}%
}}}}
\put(2611,-1636){\makebox(0,0)[lb]{\smash{{\SetFigFont{10}{12.0}{\familydefault}{\mddefault}{\updefault}{\color[rgb]{0,0,0}$\Ae(\Omega)$}%
}}}}
\put(2791,-3841){\makebox(0,0)[rb]{\smash{{\SetFigFont{10}{12.0}{\familydefault}{\mddefault}{\updefault}{\color[rgb]{0,0,0}$\frac{1}{p_1}$}%
}}}}
\put(3601,-3841){\makebox(0,0)[lb]{\smash{{\SetFigFont{10}{12.0}{\familydefault}{\mddefault}{\updefault}{\color[rgb]{0,0,0}$1$}%
}}}}
\put(4096,-3841){\makebox(0,0)[b]{\smash{{\SetFigFont{10}{12.0}{\familydefault}{\mddefault}{\updefault}{\color[rgb]{0,0,0}$\frac1p$}%
}}}}
\put(3646,-1501){\makebox(0,0)[lb]{\smash{{\SetFigFont{10}{12.0}{\familydefault}{\mddefault}{\updefault}{\color[rgb]{0,0,0}$s=\frac{\nd}{p}$}%
}}}}
\put(2071,-3481){\makebox(0,0)[lb]{\smash{{\SetFigFont{10}{12.0}{\familydefault}{\mddefault}{\updefault}{\color[rgb]{0,0,0}$\Az(\Omega)$}%
}}}}
\put(3196,-2716){\makebox(0,0)[b]{\smash{{\SetFigFont{10}{12.0}{\familydefault}{\mddefault}{\updefault}{\color[rgb]{0,0,0}compact}%
}}}}
\put(2386,-2941){\makebox(0,0)[lb]{\smash{{\SetFigFont{10}{12.0}{\familydefault}{\mddefault}{\updefault}{\color[rgb]{1,0,0}$\id_\Omega$}%
}}}}
\put(1936,-4291){\makebox(0,0)[b]{\smash{{\SetFigFont{10}{12.0}{\familydefault}{\mddefault}{\updefault}{\color[rgb]{0,0,0}nuclear}%
}}}}
\end{picture}%
\hfill~\\
~\hspace*{\fill}\unterbild{fig-2}
\end{minipage}
\smallskip~

Secondly we need the counterpart of this result for weighted spaces $\As (\rn, w_\alpha)$ as introduced in Remark~\ref{R2.3} with
$w_\alpha$ as in \eqref{2.9}. Let $p_1, p_2, q_1, q_2$ and $s_1, s_2$ be as in \eqref{3.3}.
Let $-\infty < \alpha_2 < \alpha_1 <\infty$. We consider  the embedding
\begin{\eq}   \label{3.6}
\id_\alpha: \quad \Ae(\rn, w_{\alpha_1} ) \hra \Az(\rn, w_{\alpha_2} )\quad\text{where}\quad \alpha=\alpha_1-\alpha_2>0.
\end{\eq}

\begin{proposition}\label{emb-w-nuc}
  Let {$1\leq p_1<\infty,\ 1\leq  p_2\leq \infty$ (with $p_2<\infty$ for $F$-spaces)}, $1\leq q_1,q_2\leq\infty$, $s_1,s_2\in\real$, and $\alpha =\alpha_1-\alpha_2>0$.
  \bli
\item
  $\id_\alpha$ given by \eqref{3.6} is compact if, and only if,
  \begin{equation}\label{3.6a}
    \frac{\alpha}{n} >  \max\left(\frac{1}{p_2}-\frac{1}{p_1},0\right)\quad\text{and}\quad \frac{s_1-s_2}{n} > \max\left(\frac{1}{p_1}-\frac{1}{p_2},0\right).
    \end{equation}
\item
  $\id_\alpha$ given by \eqref{3.6} is nuclear if, and only if,
  \begin{equation}\label{3.7}
    \frac{\alpha}{n} > 1 +  \min\left(\frac{1}{p_2}-\frac{1}{p_1},0\right)\quad\text{and}\quad \frac{s_1-s_2}{n} > 1 + \min\left(\frac{1}{p_1}-\frac{1}{p_2},0\right).
    \end{equation}
  \eli
\end{proposition}

For the compactness result (i) we refer to \cite[Thm.~2.3]{HT1},  \cite[Thm. {and Rem.}~4.2.3]{ET} (in the context of so-called admissible weights) and \cite[Prop.~2.8]{HaSk08} (in the context of Muckenhoupt weights). 
The nuclearity part (ii) is covered by \cite[Theorem 3.12, p.\,14]{HaS20} combined with the lifting \eqref{2.16}, see also \cite[Cor.~3.15, {p.22}]{HaS20}.


\subsection{Main assertion}    \label{S3.2}
\blue{We first restrict ourselves to the non-limiting situation, that is, we assume 
$1<p,q<\infty$. We consider the limiting cases when $p,q\in \{1,\infty\}$ in Section~\ref{S3.3} below.} Let $s\in \real$. Then
\begin{equation}   \label{3.8} 
\As (\rn)' = A^{-s}_{p',q'} (\rn), \qquad 
\frac{1}{p} + \frac{1}{p'} = \frac{1}{q} + \frac{1}{q'} =1,
\end{equation}
is the well--known duality in the framework of the dual pairing 
$\big( \SRn, \SpRn \big)$, cf. \cite[Theorem 2.11.2, p.\,178]{T83}.
Let $\Ft$ be the Fourier transform as introduced in Section~\ref{S2.1} and let $f\in \As (\rn)$ be expanded according to \eqref{2.24}.
Then
\begin{\eq}   \label{3.9}
\begin{aligned}
\Ft f &= \sum_{j=0}^\infty \sum_{G\in G^j} \sum_{m\in \zn} 2^{jn} \big(\Ft f,  \psi^j_{G,m} \big) \, \psi^j_{G,m} \\
&= \sum_{j=0}^\infty \sum_{G\in G^j} \sum_{m\in \zn} 2^{jn} \big(f, \Ft \psi^j_{G,m} \big) \, \psi^j_{G,m}
\end{aligned}
\end{\eq}
follows from $\Ft' = \Ft$ in the context of the dual pairing $\big( \SRn, \SpRn \big)$ and $(\Ft f, \psi^j_{G,m} ) = (f, \Ft\psi^j_{G,m})$
what can be justified by \eqref{3.8} and the properties of the wavelets $\psi^j_{G,m}$.

Our main result in this paper reads as follows.

\begin{theorem}   \label{T3.1}
Let 
$1<p,q_1,q_2<\infty$  and let $s_1 \in \real$, $s_2 \in \real$. Then
\begin{\eq}    \label{3.10}
\Ft: \quad A^{s_1}_{p, q_1} (\rn) \hra A^{s_2}_{p, q_2} (\rn)
\end{\eq}
is nuclear if, and only if, both
\begin{\eq}   \label{3.11}
s_1 >
\begin{cases}
n &\text{for $1<p \le 2,$} \\
\frac{2n}{p} &\text{for $2<p <\infty,$}
\end{cases}
\quad and \quad
s_2 <
\begin{cases}
-2n (1 - \frac{1}{p}) &\text{for $1<p \le 2$}, \\
-n & \text{for $2<p <\infty.$}
\end{cases} 
\end{\eq}
\end{theorem}

\blue{
  \begin{remark}
Note that \eqref{3.11} can also be written as 
$ s_1 > n- \tau^{n+}_{p'}$ {and} $s_2<-n-\tau^{n-}_{p'} $ {with $\tau^{n+}_{p'}$ and $\tau^{n-}_{p'}$ as in \eqref{2.28} with \eqref{2.27}, replacing $p$ by $p'$ and} using $\frac1p+\frac{1}{p'}=1$. We return to this observation in Remark~\ref{R-thm-nuc} below. There one also finds some illustration of the corresponding parameter areas in Figure~\ref{fig-3}. {This discussion will be extended in Remark~\ref{R-last} to the limiting cases $p=1$ and $p=\infty$ where compactness and nuclearity coincide.}
\end{remark}}

\begin{proof}
{\em Step 1.} First we prove that \eqref{3.11} ensures that $\Ft$ in \eqref{3.10} is nuclear.
By elementary embeddings (monotonicity of the spaces $\As (\rn)$ with respect to $s$)
  and the ideal property of $\Nc$ it is sufficient to deal with
\begin{\eq}   \label{3.12}
\Ft: \quad H^{s_1}_p (\rn) \hra H^{s_2}_p (\rn), \qquad 1<p<\infty,
\end{\eq}
where $H^s_p (\rn)$ are the Sobolev spaces according to \eqref{2.12}, \eqref{2.13}, normed by
\begin{\eq}    \label{3.13}
\| f \, | H^s_p (\rn) \| = \| (w_s \wh{f})^\vee \, | L_p (\rn) \|, \qquad 1<p<\infty, \quad s\in \real,
\end{\eq}
with $w_s (x) = (1 + |x|^2 )^{s/2}$, $x\in \rn$. \\

{\em Step 2.} Let $1<p \le 2$. We rely on \eqref{3.9}. By \eqref{2.26} one has
\begin{\eq}   \label{3.14}
2^{jn} \| \psi^j_{G,m} \, | H^{s_2}_p (\rn) \| \sim 2^{j(s_2 +n - \frac{n}{p})}, \qquad j \in \no, \quad m\in \zn.
\end{\eq}
It follows from the duality \eqref{3.8} and \eqref{2.12} that $H^{s_1}_p (\rn)' = H^{-s_1}_{p'} (\rn)$. Then one obtains from
\eqref{3.13} and the Hausdorff--Young inequality \eqref{1.3} that
\begin{\eq}   \label{3.15}
\begin{aligned}
\| \Ft \psi^j_{G,m} \, | H^{-s_1}_{p'} (\rn) \| &\le c \, \| \Ft \big( w_{-s_1} \psi^j_{G,m} \big) \, | L_{p'} (\rn) \| \\
&\le c' \, \| w_{-s_1} \psi^j_{G,m} \, | L_p (\rn) \| \\
&\le c'' (1 + 2^{-j} |m| )^{-s_1} \, 2^{-j \frac{n}{p}},
\end{aligned}
\end{\eq}
$j \in \no$, $m\in \zn$. Then \eqref{3.9}, \eqref{3.14}, \eqref{3.15} applied to \eqref{3.1}, \eqref{3.2} show that
\begin{\eq}   \label{3.16}
\begin{aligned}
\| \Ft \, | \Nc \big(H^{s_1}_p (\rn), H^{s_2}_p (\rn) \big) \| & \le c \, \sum^\infty_{j=0} \sum_{m \in \zn} (1 + 2^{-j} |m|)^{-s_1}
2^{j(s_2 +n - \frac{2n}{p})} \\
&\le c' \sum^\infty_{j=0} 2^{j(s_2 +n- \frac{2n}{p})} \Big( \sum_{|m| \le 2^j} 1 + \sum^\infty_{k=1} 2^{-k s_1} 2^{(j+k)n} \Big)  \\
&\le c'' \, \sum^\infty_{j=0} 2^{js_2 + j 2n (1- \frac{1}{p})} \, \sum^\infty_{k=0} 2^{-k(s_1 -n)} <\infty
\end{aligned}
\end{\eq}
if both $s_1 >n$ and $s_2 < -\frac{2n}{p'}$. This proves that $\Ft$ is nuclear as claimed in \eqref{3.11} for $1<p \le 2$.\\

{\em Step 3.} Let $2 < p <\infty$. As recalled in Section~\ref{S3.1} the operator ideal $\Nc$ is symmetric. Then one obtains from the
above--mentioned duality for $H^s_p (\rn)$ and $\Ft= \Ft'$ that
\begin{\eq}   \label{3.17}
\Ft: \quad H^{s_1}_p (\rn) \hra H^{s_2}_p (\rn), \qquad 2 < p <\infty,
\end{\eq}
is nuclear if, and only if,
\begin{\eq}    \label{3.18}
\Ft: \quad H^{-s_2}_{p'} (\rn) \hra H^{-s_1}_{p'} (\rn), \qquad \frac{1}{p} + \frac{1}{p'} =1,
\end{\eq}
is nuclear. So it follows from Step 2 that $\Ft$ is nuclear as claimed in \eqref{3.11} for $2<p<\infty$.\\

{\em Step 4.} We prove in two steps that the conditions \eqref{3.11} are also necessary to ensure that $\Ft$ in \eqref{3.10} is nuclear.
Let $1<p \le 2$ and let
\begin{\eq}   \label{3.19}
\Ft: \quad A^{s_1}_{p', q_1} (\rn) \hra A^{s_2}_{p', q_2} (\rn), \qquad \frac{1}{p} + \frac{1}{p'} =1,
\end{\eq}
be nuclear. \blue{According to {\eqref{3.11}} with $1<p\leq 2$ replaced by $2\leq p'<\infty$ we wish to prove that $s_2<-n$ and $s_1>2n(1-\frac1p)$.
} By the ideal property of $\Nc$ it is sufficient for the proof of $s_2 <-n$ to deal with
\begin{\eq}   \label{3.20}
\Ft: \quad W^k_{p'} (\rn) \hra A^{s_2}_{p', q_2} (\rn), \qquad 2\le p' <\infty, \quad k\in \nat,
\end{\eq}
where $W^k_{p'} (\rn)$ are the classical Sobolev spaces. 
Let $f\in L_p (\rn, w_k)$ according to \eqref{2.14}. Then it follows from the Hausdorff--Young inequality \eqref{1.3} and
\begin{\eq}   \label{3.21}
\Dd^\alpha \Fti f (x) = i^{|\alpha|} \, \Fti \big(\xi^\alpha f (\xi)\big) (x) \in L_{p'} (\rn), \qquad |\alpha| \le k,
\end{\eq}
that $\Fti: \, L_p (\rn, w_k) \hra W^k_{p'} (\rn)$. Thus $\Ft \Fti = \id$ combined with \eqref{3.19}, specified by \eqref{3.20}, shows that
\begin{\eq}   \label{3.22}
\id: \quad L_p (\rn, w_k) \hra A^{s_2}_{p', q_2} (\rn)
\end{\eq}
is nuclear. But now $s_2<-n$ is an immediate consequence of Proposition~\ref{emb-w-nuc}(ii) applied to \eqref{3.22}.


\ignore{Now one obtains from the well--known restriction and extension properties of the above function spaces and
again the ideal property of $\Nc$ that
\begin{\eq}  \label{3.23}
\id: \quad L_p (U) \hra A^{s_2}_{p', q_2} (U), \qquad U = \{x\in \rn: \, |x| <1 \},
\end{\eq}
is also nuclear. Then $s_2 <-n$ in \eqref{3.11} for $2\le p <\infty$ (now changing the roles of $p$ and $p'$)
follows from \eqref{3.4}, \eqref{3.5}.}

The justification  of
$s_1 >n$ for $1<p \le 2$ in \eqref{3.11} is now a matter of duality similarly as in Step 3.\\

{\em Step 5.} We deal with the remaining cases. { First we prove that $\Ft$ in \eqref{3.10} cannot be a nuclear mapping if $2\le p <\infty$ and $s_1 \le 2n/p$.  It should be clear that it is sufficient to prove it for $s_1=2n/p$. }  Theorem~\ref{Thm-comp} in Section~\ref{S2.3} implies that 
\begin{\eq}   \label{3.24}
	\Ft: \quad B^{\frac{2n}{p}}_{p,q} (\rn) \hra {B^s_{p,p}} (\rn), \qquad 2\le p<\infty, \quad 1<q<\infty,
\end{\eq}
with
\begin{\eq}  \label{3.25}
	s \le d^n_p = 2n \big( \frac{1}{p} - \frac{1}{2} \big) = \frac{2n}{p} - n,
\end{\eq}
is continuous {where we benefit from the fact that $p<\infty$  and we can thus choose $t$ with $0<t<\frac{2n}{p}$ and apply elementary embeddings such that $B^{\frac{2n}{p}}_{p,q} (\rn) \hra B^t_{p,p}(\rn)\hra {B^s_{p,p}} (\rn)$ for arbitrary $q$.
}


 The ideal
property of $\Nc$ and the elementary embeddings {$B^{\frac{2n}{p}}_{p,q_0}(\rn)\hra A^{{\frac{2n}{p}}}_{p,q_1}(\rn)$ with $q_0=\min(p,q_1)$, and $A^{s_2}_{p,q_2}(\rn)\hra B^s_{p,p}(\rn)$ with $s<s_2$} 
show that nuclearity of  the operator \eqref{3.10} implies the nuclearity the operator \eqref{3.24}. So it is sufficient  to prove that the continuous mapping 
\eqref{3.24}, \eqref{3.25} is not nuclear. 
Let $I_\alpha$ be the lift according to \eqref{2.10}, \eqref{2.11} with $w_\alpha (x) =
(1+|x|^2)^{\alpha/2}$ as in \eqref{2.9}. Let $W_\alpha$, $W_\alpha f = w_\alpha f$, be the related multiplication operator. Based on
what we already know we factorise
\begin{\eq}   \label{3.26}
	W_\alpha : \quad B^{\frac{2n}{p}}_{p,q} (\rn) \hra B^{d^n_p}_{p,p} (\rn)
\end{\eq}
by the continuous mappings 
\begin{\eq}   \label{3.27}
	B^{\frac{2n}{p}}_{p,q} (\rn) \os{\Ft}{\hra} {B^s_{p,p}} (\rn) \os{I_\alpha}{\hra} B^\sigma_{p,p} (\rn) \os{\Fti}{\hra} B^{d^n_p}_{p,p} (\rn),
\end{\eq}
where $\sigma >0$ and $\alpha = s -\sigma$.
If we assume, in addition, that $\Ft$ in \eqref{3.24} is nuclear, then $W_\alpha$ is also nuclear. Under this assumption it follows from
$\id = W_\alpha \circ W_{-\alpha}$ and the isomorphism \eqref{2.15} that
\begin{\eq}   \label{3.28}
	\id: \quad B^{\frac{2n}{p}}_{p,q} (\rn, w_{-\alpha}) \hra B^{d^n_p}_{p,p} (\rn)
\end{\eq}
is also nuclear, where $-\alpha = \sigma -s >0$. Furthermore one has by \eqref{3.25} that
$\frac{2n}{p} - d^n_p =n$. But this contradicts \eqref{3.7}. This shows that $s_1 > 2n/p$ is necessary to ensure that $\Ft$ in 
\eqref{3.10} is nuclear if $2\le p <\infty$. The corresponding assertion for $1<p \le 2$ is again a matter of duality based on
\eqref{3.8}, similarly as in \eqref{3.17}, \eqref{3.18}.
\ignore{
	{\em Step 5.} We deal with the remaining cases. It follows from Theorem~\ref{Thm-comp} in Section~\ref{S2.3} that
\begin{\eq}   \label{3.24}
\Ft: \quad B^{\frac{2n}{p}}_{p,q} (\rn) \hra {B^s_{p,p}} (\rn), \qquad 2\le p<\infty, \quad 1<q<\infty,
\end{\eq}
with
\begin{\eq}  \label{3.25}
s \le d^n_p = 2n \big( \frac{1}{p} - \frac{1}{2} \big) = \frac{2n}{p} - n,
\end{\eq}
is continuous {where we benefit from the fact that $p<\infty$  and we can thus choose $\sigma$ with $0<\sigma<\frac{2n}{p}$ and apply elementary embeddings such that $B^{\frac{2n}{p}}_{p,q} (\rn) \hra B^\sigma_{p,p}(\rn)\hra {B^s_{p,p}} (\rn)$ for arbitrary $q$.
}


We wish to prove that $\Ft$ in \eqref{3.10} cannot be a nuclear mapping if $2\le p <\infty$ and $s_1 \le 2n/p$. The ideal
property of $\Nc$ and the elementary embeddings {$B^{\frac{2n}{p}}_{p,q_0}(\rn)\hra A^{s_1}_{p,q_1}(\rn)$ with $q_0=\min(p,q_1)$, and $A^{s_2}_{p,q_2}(\rn)\hra B^s_{p,p}(\rn)$ with $s<s_2$} 
show that it is sufficient for this purpose to prove that the continuous mapping 
\eqref{3.24}, \eqref{3.25} is not nuclear. Let $I_\alpha$ be the lift according to \eqref{2.10}, \eqref{2.11} with $w_\alpha (x) =
(1+|x|^2)^{\alpha/2}$ as in \eqref{2.9}. Let $W_\alpha$, $W_\alpha f = w_\alpha f$, be the related multiplication operator. Based on
what we already know we factorise
\begin{\eq}   \label{3.26}
W_\alpha : \quad B^{\frac{2n}{p}}_{p,q} (\rn) \hra B^{d^n_p}_{p,p} (\rn)
\end{\eq}
by the continuous mappings 
\begin{\eq}   \label{3.27}
B^{\frac{2n}{p}}_{p,q} (\rn) \os{\Ft}{\hra} {B^s_{p,p}} (\rn) \os{I_\alpha}{\hra} B^\sigma_{p,p} (\rn) \os{\Fti}{\hra} B^{d^n_p}_{p,p} (\rn),
\end{\eq}
where $\sigma >0$ and $\alpha = s -\sigma$.
If we assume, in addition, that $\Ft$ in \eqref{3.24} is nuclear, then $W_\alpha$ is also nuclear. Under this assumption it follows from
$\id = W_\alpha \circ W_{-\alpha}$ and the isomorphism \eqref{2.15} that
\begin{\eq}   \label{3.28}
\id: \quad B^{\frac{2n}{p}}_{p,q} (\rn, w_{-\alpha}) \hra B^{d^n_p}_{p,p} (\rn)
\end{\eq}
is also nuclear, where $-\alpha = \sigma -s >0$. Furthermore one has by \eqref{3.25} that
$\frac{2n}{p} - d^n_p =n$. But this contradicts \eqref{3.7}. This shows that $s_1 > 2n/p$ is necessary to ensure that $\Ft$ in 
\eqref{3.10} is nuclear if $2\le p <\infty$. The corresponding assertion for $1<p \le 2$ is again a matter of duality based on
\eqref{3.8}, similarly as in \eqref{3.17}, \eqref{3.18}.
}
\end{proof}

\noindent\begin{minipage}{0.5\textwidth}
  \begin{remark}\label{R-thm-nuc}
    In the figure aside we sketched in the usual $(\frac1p,s)$-diagram the parameter areas where the Fourier operator $\Ft$ is nuclear -- as a proper subdomain of the compactness area, recall Figure~\ref{fig-1}. Note that, using the notation \eqref{2.28} with \eqref{2.27}, one could rewrite the condition \eqref{3.11} for the nuclearity of $\Ft$ in \eqref{3.10} as well as for the compactness in \eqref{2.33} as: $\Ft$ is compact, if 
    \[
      s_1> \tau^{n+}_p\quad\text{and}\quad s_2<\tau^{n-}_p
    \]
    and $\Ft$ is nuclear, if, and only if, 
    \[
       s_1 > n- \tau^{n+}_{p'}\quad\text{and}\quad s_2<-n-\tau^{n-}_{p'} .
      \]
      This explains somehow the reflected and shifted `nuclear' parameter areas compared with the compactness areas, see also Figure~\ref{fig-2}.

    \end{remark}\vfill~
\end{minipage}\hfill\begin{minipage}{0.45\textwidth}
  ~\hfill\begin{picture}(0,0)%
\includegraphics{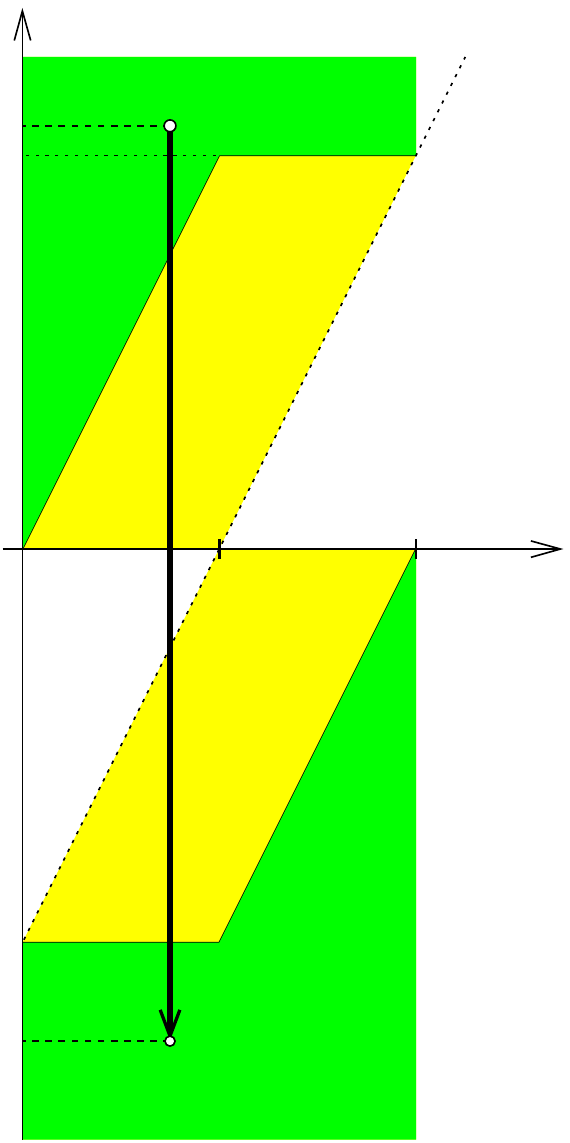}%
\end{picture}%
\setlength{\unitlength}{4144sp}%
\begingroup\makeatletter\ifx\SetFigFont\undefined%
\gdef\SetFigFont#1#2#3#4#5{%
  \reset@font\fontsize{#1}{#2pt}%
  \fontfamily{#3}\fontseries{#4}\fontshape{#5}%
  \selectfont}%
\fi\endgroup%
\begin{picture}(2589,5220)(349,-5023)
\put(496,-1051){\makebox(0,0)[lb]{\smash{{\SetFigFont{10}{12.0}{\familydefault}{\mddefault}{\updefault}{\color[rgb]{0,0,0}nuclear}%
}}}}
\put(1351,-2761){\makebox(0,0)[lb]{\smash{{\SetFigFont{10}{12.0}{\familydefault}{\mddefault}{\updefault}{\color[rgb]{0,0,0}compact}%
}}}}
\put(1351,-781){\makebox(0,0)[lb]{\smash{{\SetFigFont{10}{12.0}{\familydefault}{\mddefault}{\updefault}{\color[rgb]{0,0,0}compact}%
}}}}
\put(1576,-3886){\makebox(0,0)[lb]{\smash{{\SetFigFont{10}{12.0}{\familydefault}{\mddefault}{\updefault}{\color[rgb]{0,0,0}nuclear}%
}}}}
\put(2386,-376){\makebox(0,0)[lb]{\smash{{\SetFigFont{10}{12.0}{\familydefault}{\mddefault}{\updefault}{\color[rgb]{0,0,0}$s=d^n_p$}%
}}}}
\put(1171,-1501){\makebox(0,0)[lb]{\smash{{\SetFigFont{10}{12.0}{\familydefault}{\mddefault}{\updefault}{\color[rgb]{0,0,0}$\Ft$}%
}}}}
\put(1351,-2131){\makebox(0,0)[b]{\smash{{\SetFigFont{10}{12.0}{\familydefault}{\mddefault}{\updefault}$\frac12$}}}}
\put(1126,-286){\makebox(0,0)[b]{\smash{{\SetFigFont{10}{12.0}{\familydefault}{\mddefault}{\updefault}$A^{s_1}_{p,q_1}$}}}}
\put(1126,-4741){\makebox(0,0)[b]{\smash{{\SetFigFont{10}{12.0}{\familydefault}{\mddefault}{\updefault}$A^{s_2}_{p,q_2}$}}}}
\put(2251,-2131){\makebox(0,0)[b]{\smash{{\SetFigFont{10}{12.0}{\familydefault}{\mddefault}{\updefault}$1$}}}}
\put(2701,-2491){\makebox(0,0)[lb]{\smash{{\SetFigFont{10}{12.0}{\familydefault}{\mddefault}{\updefault}$\frac1p$}}}}
\put(1081,-2491){\makebox(0,0)[rb]{\smash{{\SetFigFont{10}{12.0}{\familydefault}{\mddefault}{\updefault}$\frac1p$}}}}
\put(406,-556){\makebox(0,0)[rb]{\smash{{\SetFigFont{10}{12.0}{\familydefault}{\mddefault}{\updefault}$n$}}}}
\put(406,-376){\makebox(0,0)[rb]{\smash{{\SetFigFont{10}{12.0}{\familydefault}{\mddefault}{\updefault}$s_1$}}}}
\put(406, 74){\makebox(0,0)[rb]{\smash{{\SetFigFont{10}{12.0}{\familydefault}{\mddefault}{\updefault}$s$}}}}
\put(406,-4606){\makebox(0,0)[rb]{\smash{{\SetFigFont{10}{12.0}{\familydefault}{\mddefault}{\updefault}$s_2$}}}}
\put(406,-4156){\makebox(0,0)[rb]{\smash{{\SetFigFont{10}{12.0}{\familydefault}{\mddefault}{\updefault}$-n$}}}}
\end{picture}%
\hfill~\\
~\hspace*{\fill}\unterbild{fig-3}
\end{minipage}
\smallskip~

\begin{remark}   \label{R3.2}
Let us briefly mention that the method from Step 5 of the proof of Theorem~\ref{T3.1}
to ensure $s_1 > 2n/p$ for $2\le p <\infty$ can also be used if $1<p \le 2$. The
counterpart of \eqref{3.24}, \eqref{3.25} is now
\begin{\eq}   \label{3.29}
\Ft: \quad {B^s_{p,p}} (\rn) \hra B^{2n(\frac{1}{p} - 1)}_{p,q} (\rn), \qquad 1<p\le 2, \quad 1<q<\infty,
\end{\eq}
with
\begin{\eq}    \label{3.30}
s \ge d^n_p = 2n \big( \frac{1}{p} - \frac{1}{2} \big) = 2n \big( \frac{1}{p}- 1 \big) +n,
\end{\eq}
again covered by Theorem~\ref{Thm-comp} in Section~\ref{S2.3}. Instead of \eqref{3.26}, \eqref{3.27} one relies now on the factorisation of
\begin{\eq}   \label{3.31}
W_\alpha : \quad B^{d^n_p}_{p,p} (\rn) \hra B^{d^n_p -n}_{p,p} (\rn)
\end{\eq}
by the continuous mappings 
\begin{\eq}   \label{3.32}
B^{d^n_p}_{p,p} (\rn) \os{\Fti}{\hra} B^\sigma_{p,q} (\rn) \os{I_\alpha}{\hra} \Bs (\rn) \os{\Ft}{\hra} B^{d^n_p -n}_{p,q} (\rn)
\end{\eq}
where $\sigma <0$ and $\alpha = \sigma -s$ with $s>d^n_p$. Afterwards one can argue as at the end of Step 5. This shows directly that
$s_2 < 2n \big( \frac{1}{p} - 1 \big)$ is necessary to  ensure that $\Ft$ in \eqref{3.10} is nuclear if $1<p \le 2$.
\end{remark}

\subsection{Limiting cases}    \label{S3.3}
So far we excluded the values $1$ and $\infty$ for the parameters $p, q_1, q_2$ in Theorem~\ref{T3.1}. \blue{We now collect what can be said about these limiting cases.}

\begin{proposition}\label{ext-q-F}
Let 
$1<p<\infty$, $1\leq q_1,q_2\leq \infty$,  and let $s_1 \in \real$, $s_2 \in \real$. Then
\begin{\eq}    \label{3.10a}
\Ft: \quad F^{s_1}_{p, q_1} (\rn) \hra F^{s_2}_{p, q_2} (\rn)
\end{\eq}
is nuclear if, and only if, \eqref{3.11} is satisfied.
\end{proposition}

\begin{proof}
{\em Step 1}.~ The sufficiency of the  assumptions \eqref{3.11} for $q_1=1$ and $q_2=\infty$  follows immediately by Theorem~\ref{T3.1} together with the elementary embeddings for the spaces $F^s_{p,q}(\rn)$. Now assume $q_1=\infty$. In case of $1< p\le 2$ we can take $\tilde{s_1}$ such that $s_1>\tilde{s_1}>n$. Then for any $\tilde{q_1}$, $1<\tilde{q_1}<\infty$, we have 
\[  F^{s_1}_{p, \infty} (\rn) \hra F^{\tilde{s_1}}_{p, \tilde{q_1}} (\rn) \stackrel{\Ft}{\hra} A^{s_2}_{p, q_2} (\rn)  , \] 
where Theorem~\ref{T3.1} ensures the nuclearity of the latter {mapping}, and thus also of \eqref{3.10a} with $q_1=\infty$. A similar argument works for $p>2$ as well as for $q_2=1$, where we always benefit from the strict inequalities in \eqref{3.11}.

{\em Step 2}.~
We prove the necessity of the condition \eqref{3.11}. If $q_1=\infty$ or $q_2=1$, then the necessity of the conditions follows once more by the elementary embeddings. {Moreover, it is sufficient to prove the necessity of the assumption concerning $s_2$ if $q_2=\infty$, and concerning $s_1$ if $q_1=1$, since the necessity of the assumption concerning the second smoothness index, in both cases,  follows once more  by elementary embeddings. }  

  If $2\le  p <\infty$ and $q_2=\infty$, then the argument that was used in Step 4 of the proof of Theorem~\ref{T3.1} shows the necessity of $s_2<-n$ in this case. \ignore{\cyan{Moreover the necessity $s_1>n$ follows from Theorem~\ref{T3.1} since the nuclearity of \eqref{3.10a} implies the nuclearity of 
\begin{\eq}    \label{3.10aa}
	\Ft: \quad F^{s_1}_{p, q_1} (\rn) \hra F^{s_3}_{p, q_3} (\rn), \qquad s_3< s_2<-n\quad \text{and} \quad 1<q_3<\infty . 
\end{\eq}
}}
If $q_2<\infty$ and $1<p\le 2$, and the mapping 
\[ \Ft: \quad F^{s_1}_{p, 1} (\rn) \hra F^{s_2}_{p, q_2} (\rn) \]
is nuclear, then,  by duality,  the mapping
\[ \Ft: \quad F^{-s_2}_{p', q'_2} (\rn) \hra F^{-s_1}_{p', \infty} (\rn) \]
is nuclear. So $-s_1<-n$ is a consequence of the last argument. 

\ignore{\cyan{ To prove that we have $s_2< -n$ it is sufficient to consider the mapping
\begin{\eq}    \label{3.10aaa}
	\Ft: \quad F^{s_3}_{p, q_3} (\rn) \hra F^{s_2}_{p, q_2} (\rn), \qquad s_3> s_2> n\quad \text{and} \quad 1<q_3<\infty . 
\end{\eq}
that is nuclear if \eqref{3.10a} is nuclear. } }

Now we  consider the case   $2<p<\infty$ and  $q_1=1$. We choose $r$ with $2<r<p$  and $s_3$ such that 
\begin{equation}\label{LS0}
	s_3=s_1+\frac{n}{r} -\frac{n}{p} . 
\end{equation} 
Then  by the Sobolev type embedding  
\begin{align}
	H^{s_3}_{r}(\rn)\hra F^{s_1}_{p,1}(\rn)\stackrel{\Ft}{\hra} F^{s_2}_{p, q_2} (\rn)  
\end{align}	
this implies that	
\begin{equation}\label{LS1}
	\Ft:\quad H^{s_3}_r(\rn)\hra F^{s_2}_{p,q_2}(\rn)
\end{equation}
 is nuclear.  
It follows from the Hausdorff--Young inequality that
\begin{\eq}   \label{LS2}
	\begin{aligned}
		\| \Ft f\, | H^{s_3}_{r} (\rn) \| \le c \, \| \Ft \big( w_{s_3} f\big) \, | L_{r} (\rn) \| 
		\le c' \, \| w_{s_3} f \, | L_{r'} (\rn) \| .
	\end{aligned}
\end{\eq}
Now $\Ft \Fti = \id$ combined with \eqref{LS1} 
shows that
\begin{\eq}   \label{LS3}
	\id: \quad L_{r'} (\rn, w_{s_3}) \hra F^{s_2}_{p, q_2} (\rn)
\end{\eq}
is nuclear if \eqref{LS1} was nuclear. Thus Proposition~\ref{emb-w-nuc}(ii) implies 
\begin{\eq}\label{LS4}
	s_3>\frac{n}{r}+\frac{n}{p}, 
\end{\eq}
which leads to  
\begin{\eq}\label{LS5}
	s_1= s_3-\frac{n}{r}+\frac{n}{p}>\frac{n}{r}+\frac{n}{p}-\frac{n}{r}+\frac{n}{p}= 
	\frac{2n}{p}. 
\end{\eq}
\ignore{\cyan{The necessity of $s_2<-n$ can be proved in the same way as in  \eqref{3.10aaa}. }}
 
Analogously we can prove the necessity in the case  $1<p<2$ and $q_2=\infty$. 
\ignore{\cyan{The necessity of $s_1>n$ can be proved in the same way as in  \eqref{3.10aa}.}}
We  choose $r$ such that $p<r<2$ and $s_3$  given by  
	\begin{equation}\label{LS5a}
	 s_3=s_2+\frac{n}{r} -\frac{n}{p} . 
\end{equation} 
Now  the Sobolev  embeddings implies   
\begin{align}
 F^{s_1}_{p,q_1}(\rn)\stackrel{\Ft}{\hra} F^{s_2}_{p, \infty} (\rn) \hra 	H^{s_3}_{r}(\rn),
\end{align}	
which implies that	
\begin{equation}\label{LS6}
	\Ft:\quad F^{s_1}_{p,q_1}(\rn) \hra H^{s_3}_r(\rn)
\end{equation}
is nuclear.  Thus in the same way as above,  
{the inequalities
\begin{\eq}   \label{LS2a}
	\begin{aligned}
	\| \Ft f \, | L_{r'} (\rn, w_{s_3}  ) \|  \le c \, \| \Ft^{-1} \big( w_{s_3} \Ft f\big) \, | L_{r} (\rn) \| 
		= c \, 	\| f\, | H^{s_3}_{r} (\rn) \|,
	\end{aligned}
\end{\eq}} 
combined with   
$\Ft \Fti = \id$ and \eqref{LS1} 
lead to the nuclearity of 
\begin{\eq}   \label{LS7}
	\id: \quad F^{s_2}_{p, q_2} (\rn) \hra  L_{r'} (\rn, w_{s_3}).
\end{\eq}
But the nuclearity of \eqref{LS7} is equivalent to the nuclearity of 
\begin{\eq}   \label{LS8}
	\id: \quad F^{s_2}_{p, q_2} (\rn, w_{-s_3}) \hra  L_{r'} (\rn),
\end{\eq}
and another application of Proposition~\ref{emb-w-nuc}(ii) implies 
\begin{\eq}\label{HT1}
	-s_3>n +\frac{n}{r'}-\frac{n}{p}. 
\end{\eq}
Consequently,  
\begin{\eq}\label{HT2}
	s_2< - 2n\left(1-\frac{1}{p} \right), 
\end{\eq}
which concludes the proof of the necessity of the conditions \eqref{3.11} in all cases.
\end{proof}

\begin{corollary}\label{ext-q-B}
  Let 
$1<p<\infty$, $1\leq q_1,q_2\leq \infty$  and let $s_1 \in \real$, $s_2 \in \real$. Then
\begin{\eq}    \label{3.10b}
\Ft: \quad B^{s_1}_{p, q_1} (\rn) \hra B^{s_2}_{p, q_2} (\rn)
\end{\eq}
is nuclear if \eqref{3.11} is satisfied. \\
Conversely, the nuclearity of \eqref{3.10b} implies \eqref{3.11} in all cases apart from $2<p<\infty$ and $q_1=1$, or $1<p<2$ and $q_2=\infty$. In case of $2<p<\infty$ and $q_1=1$ the nuclearity of \eqref{3.10b} implies $s_1\ge\frac{2n}{p}$, while in case of $1<p<2$ and $q_2=\infty$ the nuclearity of \eqref{3.10b} implies $s_2\le -2n(1-\frac{1}{p})$.
\end{corollary}

\begin{proof}
{\em Step 1}.~ The sufficiency of the  assumptions \eqref{3.11} for $q_1=1$ and $q_2=\infty$  can be proved in exactly the same way as in Proposition~\ref{ext-q-F}, Step~1 of its proof.

{\em Step 2}.~ As for the necessity in case of $q_1=1$, $1<p\leq 2$, or $q_2=\infty$ and $2\leq p<\infty$, we can follow the same arguments as presented at the beginning of Step~2 in the proof of Proposition~\ref{ext-q-F}.

{The remaining cases follow from the elementary embeddings $B^{s+\ve}_{p,p}(\rn)\hra \Bs(\rn) \hra B^{s-\ve}_{p,p}(\rn)$, $1<p<\infty$, $1\leq q\leq\infty$, $\ve>0$, and Theorem~\ref{T3.1}.}
\ignore{It remains to consider the cases $2<p<\infty$ and $q_1=1$, or $1<p<2$ and $q_2=\infty$. Let first $2<p<\infty$ and $q_1=1$ and choose $r$ with $2<r<p$ and $ s_3=s_2+\frac{n}{r} -\frac{n}{p}+\varepsilon$, $\varepsilon>0$,  replacing \eqref{LS0}. Arguing as above we get $s_1>\frac{2n}{p}-\varepsilon$. So taking the infimum over $\varepsilon>0$ we get   $s_1\ge\frac{2n}{p}$. The case $1<p<2$ and $q_2=\infty$ can be proved analogously.     }
\end{proof}

\ignore{The above theorem holds almost unchanged  $q_1$ and $q_2$ equal $1$ or $\infty$. More precisely the sufficiency part of Theorem~\ref{T3.1} holds also  if $1\le q_1\le \infty$ or $1\le q_2\le \infty$. Moreover the conditions \eqref{3.11} are necessary also for $F^s_{p,q}(\rn)$ scale i.e. if $A^{s_1}_{p,q_1}(\rn)=F^{s_1}_{p,q_1}(\rn)$ and $A^{s_2}_{p,q_2}(\rn)=F^{s_2}_{p,q_2}(\rn)$, $1\le q_1,q_2\le \infty$.  
If $A^{s_1}_{p,q_1}(\rn)=B^{s_1}_{p,q_1}(\rn)$ and $A^{s_2}_{p,q_2}(\rn)=B^{s_2}_{p,q_2}(\rn)$, $1\le q_1,q_2\le \infty$ we can prove that  the conditions  \eqref{3.11}   are necessary, except for two cases $2<p<\infty$ and $q_1=1$ or $1<p<2$ and $q_2=\infty$. If $A^{s_1}_{p,q_1}(\rn)=B^{s_1}_{p,1}(\rn)$, $2<p<\infty$ it can be proved that $s_1$ should satisfied the following inequality $s_1\ge\frac{2n}{p}$.  Analogously if $A^{s_2}_{p,q_2}(\rn)=B^{s_2}_{p,\infty}(\rn)$, $1<p<2$ Then  $s_1$ should  satisfied the  $s_2\le -2n(1-\frac{1}{p})$. 

 The above extension of Theorem~\ref{3.1} can be proved as follows.      }


Next we consider the case $p=1$. If $1\leq q_1,q_2<\infty$, we can extend Theorem~\ref{T3.1} in the desired way. 

\begin{proposition}\label{Prop3.3}
	Let 
	$1\le q_1,q_2<\infty$  and let $s_1,s_2 \in \real$. Then
	\begin{\eq}    \label{HT3}
		\Ft: \quad A^{s_1}_{1, q_1} (\rn) \hra A^{s_2}_{1, q_2} (\rn)
	\end{\eq}
	is nuclear if, and only if,
	\begin{\eq}   \label{LS11}
		s_1 > n \qquad \text{and}\qquad s_2<0 . 
	\end{\eq}
\end{proposition}

\begin{proof}
	It is sufficient to consider the Besov spaces, i.e.,
	\begin{\eq}   \label{LS12}
		\Ft: \quad B^{s_1}_{1,q_1} (\rn) \hra B^{s_2}_{1,q_2} (\rn), \qquad 1\le q_1,q_2<\infty. 
	\end{\eq}
{By \cite[Thm.~2.11.2, p.178]{T83}} 
 we have the following duality 
	\[ B^{s_1}_{1,q_1} (\rn)' = B^{-s_1}_{\infty,q_1'} (\rn)\] 
	and according to \eqref{2.26} the estimates for the norms of the wavelets  
	\begin{\eq}   \label{LS13}
		2^{jn} \| \psi^j_{G,m} \, | B^{s_2}_{1,q_2} (\rn) \| \sim 2^{j s_2}, \qquad j \in \no, \quad m\in \zn, \quad G\in G^j.
	\end{\eq} 
 
	 Using the lift property for Besov spaces and  continuity properties of the Fourier transform acting into  Besov spaces, cf.  \cite[Theorem 1]{Tai},  one  obtains 	
	 \begin{align}   \label{LS14}
			\| \Ft \psi^j_{G,m} \, | B^{-s_1}_{\infty,q_1'} (\rn) \|  & \le \| I_{-s_1}\Ft\psi^j_{G,m}| B^{0}_{\infty,q_1'} (\rn) \|  \\ \nonumber
			&\le \| \Ft w_{-s_1} \psi^j_{G,m} \, | B^{0}_{\infty, 1} (\rn) \|    
			\le c \, \| w_{-s_1} \psi^j_{G,m} \, | L_1 (\rn) \| \\ \nonumber
			& \le c (1 + 2^{-j} |m| )^{-s_1} \, 2^{-j n}, 
	\end{align}
	$j \in \no$, $m\in \zn$. Then \eqref{3.9}, \eqref{LS13}, \eqref{LS14} applied to \eqref{3.1}, \eqref{3.2} show {in the same way as in \eqref{3.16}} that
	\begin{\eq}   \label{HT4}
		\begin{aligned}
			\| \Ft \, | \Nc \big(B^{s_1}_{1,q_1} (\rn), B^{s_2}_{1,q_2} (\rn) \big) \| & \le c \, \sum^\infty_{j=0} \sum_{m \in \zn} (1 + 2^{-j} |m|)^{-s_1}
			2^{j(s_2 -n)} \\
			&\le c'' \, \sum^\infty_{j=0} 2^{js_2} \, \sum^\infty_{k=0} 2^{-k(s_1 -n)} <\infty
		\end{aligned}
	\end{\eq}
	if both $s_1 >n$ and $s_2 < 0$. This proves that $\Ft$ is nuclear as claimed.  
	 
	 Now we assume that $\Ft$ given by \eqref{HT3} is a nuclear operator. Using the continuity of the Fourier transform defined on Besov spaces, cf.  \cite[Theorem 1]{Tai},  we get
	  \begin{align}   \label{HT5}
	 	\| w_{s_2}\Ft f\, | L_\infty(\rn) \|  & = \| \Ft \Fti w_{s_2}\Ft f|  L_\infty(\rn) \| \\
	 	& \le  \| I_{s_2}f| B^{0}_{1,\infty} (\rn) \| \le c 
	 	\| f| B^{s_2}_{1,\infty} (\rn) \|. 
	 	\nonumber
	 \end{align} 
We combine {\eqref{LS12}} 
and \eqref{HT5} with the identity $\id=\Ft\circ \Fti$ and get the following nuclear embedding 
 	\begin{\eq}   \label{HT6}
 	\id:\; B^{s_1}_{1,q_1} (\rn) \stackrel{\Ft}{\hra} B^{s_2}_{1,q_2} (\rn) \stackrel{\Ft}{\hra} L_\infty(w_{s_2},\rn).
 \end{\eq}
Thus, the embedding 
 \[ \id: \; B^{s_1}_{1,q_1} (\rn, w_{-s_2}) \hookrightarrow B^0_{\infty,\infty}(\rn)\]
 is also  nuclear, and another application of Proposition~\ref{emb-w-nuc}(ii) implies $-s_2>0$ and $s_1>n$. 
\end{proof}


\begin{remark}
We would like to mention that one can also use a more direct argument to prove the above extensions. This would be based on the modifications
\begin{\eq}   \label{3.50}
\Ft: \quad L_1 (\rn) \hra B^0_{\infty,1} (\rn)
\end{\eq}
and
\begin{\eq}   \label{3.51}
\Ft: \quad B^0_{1, \infty} (\rn) \hra L_\infty (\rn)
\end{\eq}
of the Hausdorff--Young mappings. Here \eqref{3.50} follows from
\begin{\eq}   \label{3.52}
\begin{aligned}
\| \Ft f \, |B^0_{\infty,1} (\rn)\| &\sim \sum^\infty_{j=0} \| \Fti (\vp_j f)\,| L_\infty (\rn) \| \\
&\le c \sum^\infty_{j=0} \| \vp_j f \, | L_1 (\rn) \| \\
&\sim \| f \, | L_1 (\rn) \|
\end{aligned}
\end{\eq}
and \eqref{3.51} from
\begin{\eq}   \label{3.53}
\begin{aligned}
\|\Ft f\,| L_\infty (\rn) \| &\sim \sup_{j\in \no} \| \vp_j \Ft f \, | L_\infty (\rn) \| \\
&\le c \, \sup_{j\in \no} \| \Fti \vp_j \Ft f \, | L_1 (\rn) \| \\
& = c \, \| f\, | B^0_{1,\infty} (\rn) \|.
\end{aligned}
\end{\eq}
\end{remark}

Finally, by duality arguments, one can cover the case $p=\infty$, when $A=B$ and $1<q_1,q_2\leq\infty$. But using Proposition~\ref{Prop3.3}, we have even a characterisation in this case.

 \begin{corollary}\label{Cor3.4}
 	Let 
 $1 < q_1,q_2\le \infty$  and $s_1,s_2 \in \real$. Then
 \begin{\eq}    \label{HT7}
 	\Ft: \quad B^{s_1}_{\infty, q_1} (\rn) \hra B^{s_2}_{\infty, q_2} (\rn)
 \end{\eq}
 is nuclear if, and only if,
 \begin{\eq}   \label{HT8}
 	s_1 > 0 \qquad \text{and}\qquad s_2<-n . 
 \end{\eq}	
 \end{corollary}

 \begin{proof}
   The sufficiency follows from Proposition~\ref{Prop3.3} in case of $A=B$ and duality, that is,  
\begin{\eq}  \label{3.58}
\Ft: \quad B^{s_1}_{\infty, q_1} (\rn) \hra B^{s_2}_{\infty, q_2} (\rn), \qquad 1<q_1, q_2 \le \infty,
\end{\eq}
is nuclear if, $s_1 >0$ and $s_2 <-n$. 

We come to the necessity. Note first, that by Proposition~\ref{Prop3.3}  
\begin{\eq}   \label{3.57}
\Ft: \quad B^{s_1}_{1, q_1} (\rn) \hra B^{s_2}_{1, q_2} (\rn), \qquad 1\le q_1, q_2 < \infty,
\end{\eq}
is nuclear if, and only if, $s_1 >n$, $s_2 <0$. Then $\Ft$ is also compact.  Conversely, if $\Ft$ in \eqref{3.57} is compact, then it
follows {from \eqref{3.6a}} by the same reduction as {in \eqref{HT6}} to weighted spaces that again $s_1 >n$, $s_2 <0$. In other words, $\Ft$ in \eqref{3.57} is nuclear
if, and only if, it is compact.

Let now, conversely, $\Ft$ in \eqref{3.58} for some $s_1 \in \real$ and $s_2 \in \real$, be
nuclear. Then for the same $s_1 \in \real$ and $s_2 \in \real$ both $\Ft$ in \eqref{3.58} and
\begin{\eq}   \label{3.59} 
\Ft: \quad \os{\circ}{B}{}^{s_1}_{\infty, q_1} (\rn) \hra \os{\circ}{B}{}^{s_2}_{\infty,q_2} (\rn), \qquad 1<q_1, q_2 \le \infty,
\end{\eq}
are compact. \blue{Here $\os{\circ}{B}{}^{s}_{\infty, q}(\rn)$ stands for the closure of {$\SRn$} 
  in $B^s_{\infty,q}(\rn)$, which is a proper subspace of $B^s_{\infty,q}(\rn)$.}
Using the duality  
\begin{\eq}   \label{3.60}
\os{\circ}{B}{}^s_{\infty, q} (\rn)' = B^{-s}_{1, q'} (\rn), \qquad s\in \real, \quad 1\le q \le \infty, \quad \frac{1}{q} + 
\frac{1}{q'} =1,
\end{\eq}
cf. \cite[Remark 2.11.2/2, p.\,180]{T83}, then
\begin{\eq}   \label{3.61}
\Ft: \quad B^{-s_2}_{1, q'_2} (\rn) \hra B^{-s_1}_{1, q'_1} (\rn), \qquad 1 \le q'_1, q'_2 <\infty,
\end{\eq}
is compact. This requires $-s_2 >n$ and $-s_1 <0$.
\end{proof}


\begin{remark}   \label{R3.7}
Note that the nuclear counterpart of the argument in \eqref{3.59} is not clear, maybe not true, as there is no projection operator from 
$\ell_\infty$ onto $c_0$, \cite[Corollary 2.5.6, p.\,46]{AlK06}, on which a related proof could be based. Furthermore, according to 
\cite[p.\,343]{Pie07} the operator ideal  $\Nc$ is not injective which would otherwise ensure the nuclear version of \eqref{3.59}.
\end{remark}

\begin{remark}\label{R-last}
{  Let us remark that 
\begin{\eq}   \label{3.57a}
\Ft: \quad B^{s_1}_{1, q_1} (\rn) \hra B^{s_2}_{1, q_2} (\rn), \qquad 1\le q_1, q_2 < \infty,
\end{\eq}
is nuclear if, and only if, it is compact. The same phenomenon can be observed for
\begin{\eq}  \label{3.58a}
\Ft: \quad B^{s_1}_{\infty, q_1} (\rn) \hra B^{s_2}_{\infty, q_2} (\rn), \qquad 1<q_1, q_2 \le \infty,
\end{\eq}
which is nuclear if, and only if, it is compact. In view of Theorems~\ref{Thm-comp} and \ref{T3.1} this is different from the situation for $1<p<\infty$, when the conditions for the nuclearity of $\Ft$ are indeed stronger than for its compactness.} In other words, for $\Ft : B^{s_1}_{p,q_1}(\rn) \hra B^{s_2}_{p,q_2}(\rn)  $  compactness and nuclearity coincide if, and only if, $p=1$ or $p=\infty$ (with appropriately chosen $q_1,q_2$), as can be also seen from the 
   reformulated conditions in Remark~\ref{R-thm-nuc} or in Figure~\ref{fig-3}. We always have $n-\tau^{n+}_{p'}\geq \tau^{n+}_p$ and $-n-\tau^{n-}_{p'}\leq \tau^{n-}_p$, with equality in case of $p=1$ or $p=\infty$. A similar phenomenon was observed in \cite[Cor.~3.16, Rem.~3.18]{HaS20} related to the situations on domains as described in Proposition~\ref{nuc-dom}, and for weighted spaces, recall Proposition~\ref{emb-w-nuc}.
   \end{remark}

\ignore{One can also give a direct proof of the argument used in \eqref{HT5}, using 
\begin{\eq}  
	\Ft: \quad B^0_{1,\infty} (\rn) \to L_\infty (\rn)
	\end{\eq}
	based on
	\begin{\eq}
	\begin{aligned}
	\|\Ft f\,|L_\infty(\rn) \| &\sim \sup_{j\in \no} \|\vp_j \Ft f \, | L_\infty (\rn)\| \\
	&\le c\, \sup_{j\in \no} \| \Fti \vp_j \Ft f\,|L_1 (\rn) \|   \\
	&=c \,\|f \, | B^0_{1,\infty} (\rn) \|.
	\end{aligned}
	\end{\eq}}

\ignore{
Based on these modifications one can modify the related arguments in the proof of Theorem~\ref{T3.1} with the following outcome.

\begin{proposition}   \label{P3.3}
As stated in Proposition~\ref{Prop3.3}.
\end{proposition}

By duality one obtains the following assertion.

\begin{corollary}   \label{C3.4}
As stated in Corollary~\ref{Cor3.4}.
\end{corollary}
}

\ignore{
One may ask to which extent Theorem~\ref{T3.1} remains valid if one admits that $q_1$ and $q_2$ may be $1$ or $\infty$. Let $1<p<\infty$ and
$1\le q_1, q_2 \le \infty$. Then it follows by elementary embedding that
\begin{\eq}   \label{3.54}
\Ft: \quad A^{s_1}_{p, q_1} (\rn) \hra A^{s_2}_{p, q_2} (\rn), \qquad 1\le q_1, q_2 \le \infty,
\end{\eq}
is nuclear if $s_1$ and $s_2$ are restricted by \eqref{3.11}. If the continuous mapping $\Ft$ in \eqref{3.54} is nuclear, then one has by the same
type of arguments and Theorem~\ref{T3.1} that
\begin{\eq}   \label{3.55}
s_1 \ge
\begin{cases}
n &\text{for $1<p \le 2,$} \\
\frac{2n}{p} &\text{for $2<p <\infty,$}
\end{cases}
\quad \text{and} \quad
s_2 \le
\begin{cases}
-2n (1 - \frac{1}{p}) &\text{for $1<p \le 2$}, \\
-n & \text{for $2<p <\infty.$}
\end{cases} 
\end{\eq}
Furthermore, the well--known inclusions for the spaces $\As (\rn)$ show that \eqref{3.11} is also necessary if $1<q_1 \le \infty$ and $1 \le q_2
<\infty$. The situation is not so clear in the remaining cases. But one has at least for the spaces $\Fs (\rn)$  with $1<p<\infty$ the following 
assertion.}

\ignore{
  \begin{proposition}  \label{P3.5}
Let $1<p<\infty$ and $1 \le q_1, q_2 \le \infty$. Let $s_1 \in \real$, $s_2 \in \real$. Then
\begin{\eq}   \label{3.56}
\Ft: \quad F^{s_1}_{p, q_1} (\rn) \hra F^{s_2}_{p, q_2} (\rn)
\end{\eq}
is nuclear if, and only if, $s_1$ and $s_2$ satisfy \eqref{3.11}.
\end{proposition}

\begin{proof}
Better rely on \cite[Corollary 3.17]{HaS20} (in extension of \eqref{3.3}--\eqref{3.5}), otherwise  as stated.
\end{proof}

\begin{remark}   \label{R3.6}
We used a special case of the so--called Franke--Jawerth embedding for $F$--spaces. There is no counterpart for the $B$--spaces. But
one can extend Step 5 of the proof of Theorem~\ref{T3.1} for $2 \le p <\infty$ to $q=1$ and its counterpart for $1<p \le 2$ as outlined
in Remark~\ref{R3.2} to $q= \infty$. This shows that Proposition~\ref{P3.5} with $B$ in place of $F$ remains valid, relying again on
\cite[Corollary 3.17]{HaS20}.
\end{remark}
}

\ignore{
We could remark that for $\Ft$ in \eqref{3.10} compactness and nuclearity coincide if, and only if, $p=1$ or $p=\infty$, as can be (also) seen from the 
reformulated conditions in Remark~\ref{R-thm-nuc} or in Figure~\ref{fig-3}. We always have $n-\tau^{n+}_{p'}\geq \tau^{n+}_p$ and $-n-\tau^{n-}_{p'}\leq \tau^{n-}_p$, with equality in case of $p=1$ or $p=\infty$. A similar phenomenon was observed in \cite[Cor.~3.16, Rem.~3.18]{HaS20}.
}

\section{Weighted spaces}   \label{S4}
Let again $\As (\rn, w_\alpha)$, $A \in \{B,F \}$, and $s,p,q$ as in Definition~\ref{D2.1} be the weighted spaces as introduced in Remark~\ref{R2.3} where we restrict ourselves to the distinguished weights
\begin{\eq}   \label{4.1}
w_\alpha (x) = \big( 1 + |x|^2)^{\alpha/2}, \qquad x \in \rn, \quad \alpha \in \real.
\end{\eq}
So far we 
concentrated  mainly on the unweighted spaces $\As (\rn)$ and used their weighted generalisations as tools caused by the specific
mapping properties  of $\Ft$. But under these circumstances it is quite natural to ask how weighted counterparts of
the main assertions obtained in the above
Section~\ref{S3} and in \cite{T21} may look like. Fortunately enough there is no need  to extend the quite substantial  machinery
underlying the related theory for the spaces $\As (\rn)$ to the weighted spaces $\As (\rn, w_\alpha)$ (what might be possible), but
there is an effective short--cut based on qualitative arguments which will be described below. We rely on the same remarkable 
properties of the spaces $\As (\rn, w_\alpha)$ which we already described in Section~\ref{S2.1} with a reference to \cite[Theorem
6.5, pp.\,265--266]{T06}. In particular, the multiplier
\begin{\eq}   \label{4.2a}
  W_\beta: \ f \mapsto w_\beta f, \qquad f\in \SpRn, \quad \beta\in\real,
\end{\eq}
is for all these spaces an isomorphic  mapping,
\begin{\eq}   \label{4.2}
  \begin{aligned}
    W_\beta \As\left(\rn, w_{\alpha+\beta}\right) & = \As(\rn, w_\alpha),\\ 
  \| w_\beta f \, | \As (\rn, w_{\alpha+\beta}) \| & \sim \| f \, | \As (\rn, w_\alpha) \|, \qquad \alpha \in \real, \quad \beta \in
  \real,
  \end{aligned}
\end{\eq}
and the lift $I_\gamma$, $\gamma\in\real$,
\begin{\eq}   \label{4.3}
I_\gamma: \quad f \mapsto \big( w_\gamma \wh{f}\, \big)^\vee = \big( w_\gamma f^\vee \big)^\wedge, \qquad f\in \SpRn, \quad \gamma \in\real,
\end{\eq}
for the spaces $\As (\rn)$
according to \eqref{2.11} generates also the isomorphic mappings
\begin{\eq}   \label{4.4}
\begin{aligned}
I_\gamma A^{s+\gamma}_{p,q} (\rn, w_\alpha) &= \As (\rn, w_\alpha), \\
 \|(w_\gamma \wh{f} )^\vee |  \As(\rn, w_\alpha) \|
&\sim \| f\, | A^{s+\gamma}_{p,q} (\rn, w_\alpha) \|,
\end{aligned}
\end{\eq}
$\alpha \in \real$, $\gamma \in \real$, $s\in \real$ and $0<p,q \le \infty$ ($p<\infty$ for $F$--spaces).

Note that by the definitions of $W_\beta$ in \eqref{4.2a} and $I_\gamma$ in \eqref{4.3},
\[
\Ft \circ W_\beta\circ I_\gamma =   I_\beta \circ W_\gamma \circ \Ft \quad \text{on}\quad \SpRn,
  \]
  which directly leads to the decomposition of $\Ft$ into 
  \begin{equation}
    \label{F-I-W}
    \Ft = W_{-\gamma} \circ I_{-\beta} \circ \Ft \circ W_\beta \circ I_\gamma \quad \text{on}\quad \SpRn.
        \end{equation}
We shall benefit from this observation below, see also Remark~\ref{R-diagram}.
        
Although not needed, it might illuminate what is going on that any $f\in \SpRn$ belongs to a suitable weighted space of the above
type. More precisely, one has for fixed $0<p,q \le \infty$ that
\begin{\eq}   \label{4.5}
\SRn = \bigcap_{\alpha \in \real, s\in \real} \Bs (\rn, w_\alpha) \quad \text{and} \quad 
\SpRn = \bigcup_{\alpha \in \real, s\in \real} \Bs (\rn, w_\alpha).
\end{\eq}
This is more or less known and may be found in \cite[(2.281), p.\,74]{T20} with a reference to \cite{Kab08} for a detailed proof.

In what follows we are not interested in generality. This may explain why we suppose as in Theorem~\ref{T3.1} that $1<p,q_1,q_2
<\infty$, whereas it is quite clear  that at least some of the arguments below apply also to a wider range of these parameters.

\begin{proposition}   \label{P4.1}
Let $1<p,q_1,q_2 < \infty$ and $s_1 \in \real$, $s_2 \in \real$. Let $-\infty <\alpha_1, \alpha_2, \beta, \gamma <\infty$ and $A \in
\{B,F \}$. Then there is a continuous mapping
\begin{\eq}   \label{4.6}
\Ft: \quad A^{s_1+\gamma}_{p,q_1} (\rn, w_{\alpha_1 + \beta}) \hra A^{s_2 +\beta}_{p,q_2} (\rn, w_{\alpha_2 +\gamma})
\end{\eq}
if, and only if, there is a continuous mapping
\begin{\eq}   \label{4.7}
\Ft: \quad A^{s_1}_{p,q_1} (\rn, w_{\alpha_1}) \hra A^{s_2}_{p,q_2} (\rn, w_{\alpha_2}).
\end{\eq}
Furthermore, $\Ft$ in \eqref{4.6} is compact if, and only if, $\Ft$ in \eqref{4.7} is compact, and $\Ft$ in \eqref{4.6} is nuclear if,
and only if, $\Ft$ in \eqref{4.7} is nuclear.
\end{proposition}

\begin{proof}
{\em Step 1.} Let $\Ft$ in \eqref{4.7} be continuous and let $f\in A^{s_1}_{p,q_1} (\rn, w_{\alpha_1 +\beta})$. Then it follows from
\eqref{4.2} that
\begin{\eq}   \label{4.8}
\| \Ft (w_\beta f) \, | A^{s_2}_{p,q_2} (\rn, w_{\alpha_2}) \| \le c \, \|w_\beta f \, | A^{s_1}_{p,q_1} (\rn, w_{\alpha_1}) \|.
\end{\eq}
By \eqref{4.3} one has
\begin{\eq}   \label{4.9}
\Ft \circ W_\beta = I_\beta \circ \Ft.
\end{\eq}
Inserted in \eqref{4.8} one obtains by \eqref{4.2} and \eqref{4.4} that
\begin{\eq}   \label{4.10}
\| \Ft f \, | A^{s_2 + \beta}_{p, q_2} (\rn, w_{\alpha_2} ) \| \le c\, \| f \, | A^{s_1}_{p,q_1} (\rn, w_{\alpha_1 +\beta}) \|.
\end{\eq}
This proves the continuity of $\Ft$ in \eqref{4.6} with $\gamma =0$. Let again $\Ft$ in \eqref{4.7} be continuous and let $f \in
A^{s_1 + \gamma}_{p, q_1} (\rn, w_{\alpha_1} )$. Then it follows from \eqref{4.4} that 
\begin{\eq} \label{4.11}
\| \Ft (I_\gamma f) \, | A^{s_2}_{p,q_2} (\rn, w_{\alpha_2})\| \le c \, \|I_\gamma f\, |A^{s_1}_{p, q_1} (\rn, w_{\alpha_1} ) \|.
\end{\eq}
By \eqref{4.3} one has
\begin{\eq}   \label{4.12}
\Ft \circ I_\gamma  = W_\gamma \circ \Ft.
\end{\eq}
Inserted in \eqref{4.11} one obtains by \eqref{4.2} and \eqref{4.4} that
\begin{\eq}   \label{4.13}
\| \Ft f\,| A^{s_2}_{p, q_2} (\rn, w_{\alpha_2 +\gamma}) \| \le c \, \|f \, | A^{s_1 +\gamma}_{p, q_1} (\rn, w_{\alpha_1})\|.
\end{\eq}
This proves the continuity of $\Ft$ in \eqref{4.6} with $\beta =0$. A {combination} of the above arguments for $\gamma =0$ and $\beta =0$
shows that $\Ft$ in \eqref{4.6} is continuous for all $\beta \in \real$ and $\gamma \in \real$ if $\Ft$ in \eqref{4.7} is continuous. 
But this covers also the reverse step from \eqref{4.6} to \eqref{4.7} and proves the above proposition as far as the continuity is
concerned.\\

{\em Step 2.} The above arguments combine supposed mapping properties for $\Ft$ with isomorphisms of type \eqref{4.2} and \eqref{4.4}.
But then not only continuity is inherited, but also compactness and nuclearity.
\end{proof}

\begin{remark}\label{R-diagram}
  The strategy of the above proof can be illustrated by the following commutative diagram:
  \[
    \begin{array}{ccccc}A^{s_1}_{p,q_1}(\rn, w_{\alpha_1}) & \xrightleftharpoons[W_\beta]{W_{-\beta}} &A^{s_1}_{p,q_1}(\rn, w_{\alpha_1+\beta})
      & \xrightleftharpoons[I_{\gamma}]{I_{-\gamma}} &A^{s_1+\gamma}_{p,q_1}(\rn, w_{\alpha_1+\beta})\\
 \Ft \Big\downarrow & & & &  \Big\downarrow \Ft\\
A^{s_2}_{p,q_2}(\rn, w_{\alpha_2})
& \xleftrightharpoons[I_{-\beta}]{I_\beta} &A^{s_2+\beta}_{p,q_2}(\rn, w_{\alpha_2})&  \xleftrightharpoons[W_{-\gamma}]{W_\gamma} &  A^{s_2+\beta}_{p,q_2}(\rn, w_{\alpha_2+\gamma})
\end{array}
  \]
Here the mappings \eqref{4.6} and \eqref{4.7} can be found on the left-hand and right-hand side of the diagram, while travelling around in the diagram is based on \eqref{F-I-W}.      
\end{remark}

Now one can extend assertions about continuity, compactness and nuclearity for the unweighted spaces $\As (\rn)$ to their weighted counterparts.

\begin{theorem}   \label{T4.2}
Let $1< p, q_1, q_2 <\infty$ and $s_1 \in \real$, $s_2 \in \real$. Let $\beta \in \real$, $\gamma \in \real$ and $A \in \{B,F \}$.
 \bli
\item
Let $d^n_p$ and $\tau^{n+}_p, \tau^{n-}_p$ be as in \eqref{2.27}, \eqref{2.28}. Then
\begin{\eq}   \label{4.14}
\Ft: \quad A^{s_1 +\gamma}_{p, q_1} (\rn, w_\beta) \hra A^{s_2 +\beta}_{p, q_2} (\rn, w_\gamma)
\end{\eq}
is compact if both $s_1 > \tau^{n+}_p$ and $s_2 < \tau^{n-}_p$. 
\item
Then $\Ft$ in \eqref{4.14} is nuclear if, and only if, both
\begin{\eq}   \label{4.15}
s_1 >
\begin{cases}
n &\text{for $1<p \le 2,$} \\
\frac{2n}{p} &\text{for $2<p <\infty,$}
\end{cases}
\quad and \quad
s_2 <
\begin{cases}
-2n (1 - \frac{1}{p}) &\text{for $1<p \le 2$}, \\
-n & \text{for $2<p <\infty.$}
\end{cases} 
\end{\eq}
\eli
\end{theorem}

\begin{proof}
This follows immediately from Proposition~\ref{P4.1} with $\alpha_1 = \alpha_2 =0$ combined with Corollary~\ref{F-comp-cor}  and Theorem~\ref{T3.1}.
\end{proof}

\begin{remark}   \label{R4.3}
It was one of the main aims of \cite{T21} to measure the degree of compactness of
\begin{\eq}   \label{4.16}
\Ft: \quad B^{s_1}_{p, q_1} (\rn) \hra B^{s_2}_{p,q_2} (\rn),
\end{\eq}
$1<p, q_1, q_2 <\infty$ and $s_1 > \tau^{n+}_p$, $s_2 <\tau^{n-}_p$ in terms of entropy numbers. Proposition~\ref{P4.1} and its proof
show that these assertions can also be extended to the compact mappings in \eqref{4.14}.
\end{remark}

\begin{remark}   \label{R4.4}
It is quite obvious that one can relax the assumptions $1<q_1, q_2 <\infty$ for the compact mappings in \eqref{4.14} by $0<q_1, q_2
\le \infty$. This applies also to related entropy numbers as mentioned in Remark~\ref{R4.3}.
\end{remark}

\bigskip\bigskip~
\small

\noindent Dorothee D. Haroske\\
\noindent  Institute of Mathematics,
Friedrich Schiller University Jena, 07737 Jena, Germany\\
\noindent {\it E-mail}:  \texttt{dorothee.haroske@uni-jena.de}

\bigskip

\noindent Leszek Skrzypczak\\
\noindent  Faculty of Mathematics and Computer Science,
Adam Mickiewicz University, Ul. Uniwersytetu Pozna\'nskiego 4, 61-614 Pozna\'n,
Poland\\
\noindent {\it E-mail}:  \texttt{lskrzyp@amu.edu.pl}

\bigskip

\noindent Hans Triebel\\
\noindent  Institute of Mathematics,
Friedrich Schiller University Jena, 07737 Jena, Germany\\
\noindent {\it E-mail}:  \texttt{hans.triebel@uni-jena.de}

\end{document}